\documentclass[a4paper,11pt]{article}

\usepackage[mac]{inputenc}
 \usepackage[T1]{fontenc}
 \usepackage[normalem]{ulem}
 \usepackage[english]{babel}
\usepackage{amsmath}
  \usepackage{amsthm}
 \usepackage{bbm}
 \usepackage{amssymb}
 \usepackage{verbatim}
 \usepackage{vmargin}
 \usepackage{graphicx}
 \usepackage{color}
 \usepackage{epsfig}
 \usepackage[only,llbracket,rrbracket]{stmaryrd}
  \usepackage{pgf}
 \usepackage{tikz} 
\usepackage{psfrag}

\newtheorem{thm}{Theorem}[section]
\newtheorem*{deft}{Definition}

\newtheorem{prop}{Proposition}[section]
\newtheorem{lem}{Lemma}[section]
\newtheorem{coro}{Corollary}[section]
\numberwithin{equation}{section}

\theoremstyle{remark}
{\vskip 0.5cm}

\newtheorem{rque}{\textbf{Remark}}[section]{\vskip 0.5cm}
{\vskip 0.5cm} 

\title{On the confinement of a tokamak plasma}
\author{Daniel Han-Kwan\footnote{Département de Mathématiques et Applications, Ecole Normale Supérieure, 45 rue d'Ulm, 75230 Paris Cedex 05, France (hankwan@dma.ens.fr)}}
\begin{document}
 \maketitle

\begin{abstract}
The goal of this paper is to understand from a mathematical point of view the magnetic confinement of plasmas for fusion.
Following Frénod and Sonnendrücker \cite{FS2}, we first use two-scale convergence tools to derive a gyrokinetic system for a plasma submitted to a large magnetic field with a slowly spatially varying intensity.  We formally derive from this system a simplified bi-temperature fluid system. We then investigate the behaviour of the plasma in such a regime and we prove nonlinear stability or instability depending on which side of the tokamak we are looking at. In our analysis, we will also point out that there exists a temperature gradient threshold beyond which one can expect stability, even in the ``bad'' side : this corresponds to the so-called H-mode.
\end{abstract}

\section{Introduction}

\subsection{Magnetic confinement for plasmas}

Fusion is undoubtly one of the most promising research fields in order to find new sources of energy. For the time being, magnetic confinement fusion represents one of the two main approaches (the other one being inertial confinement fusion). The principle consists basically in using a magnetic field in order to confine the very high temperature plasma.
Good confinement is absolutely compulsory since the plasma could otherwise damage the surrounding materials. 
\\
\begin{center}
\psfrag{$r$}{$r$}
\psfrag{0}{$0$}
\psfrag{$B$}{$B$}
\psfrag{a}{$e_\varphi$}
\psfrag{b}{$e_r$}
\psfrag{c}{$e_\theta$}
\psfrag{d}{$\theta$}

\includegraphics[scale=0.4]{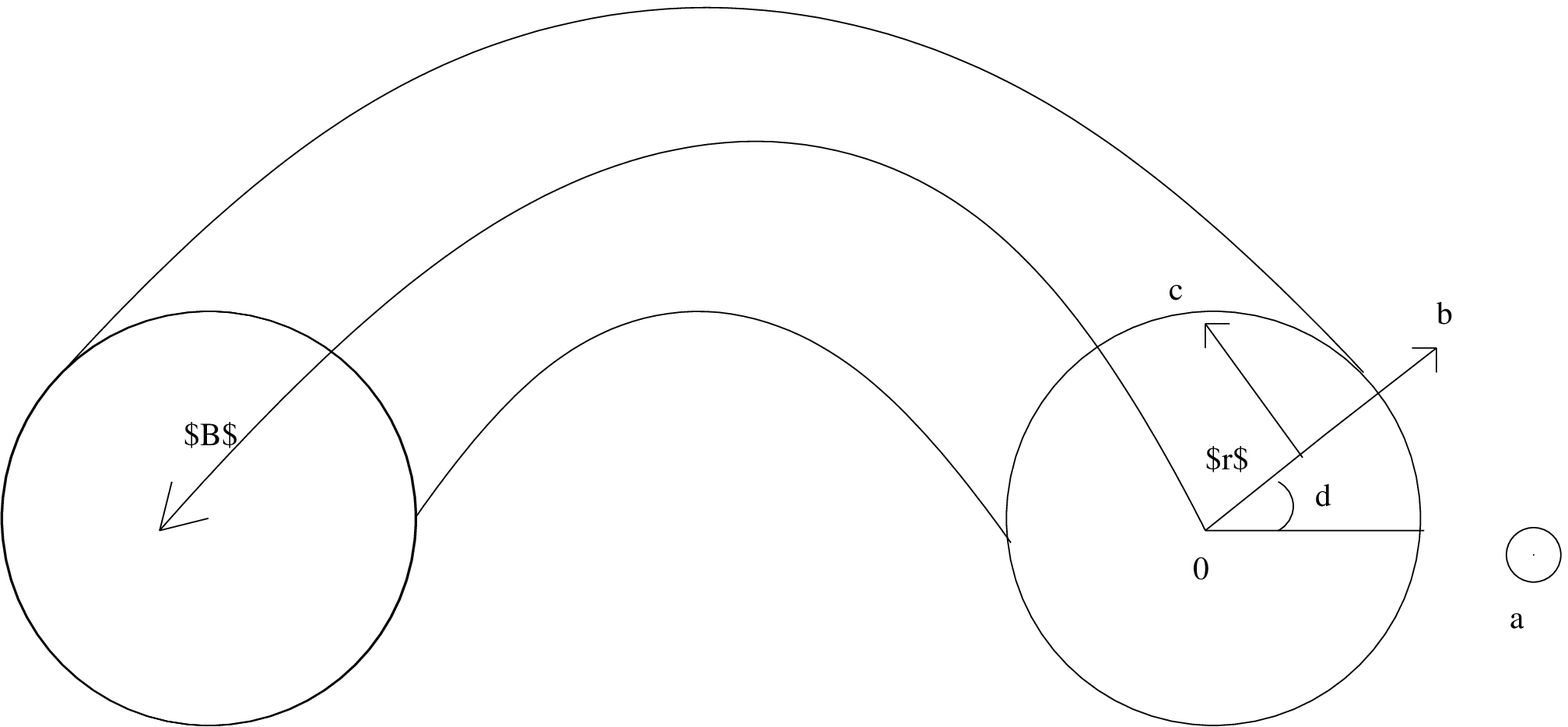}
\end{center}

A first step towards confinement is to use a tokamak\footnote{Actually there are other possibilities, like stellarators. These kinds of devices are much more difficult to study from the mathematical viewpoint, since they have a very complex structure.}, i.e. a torus-shaped box and consider a large purely toroidal magnetic field $B$, in other words $B=\frac{\mathcal{B}}{\epsilon} e_\varphi$ with $\epsilon >0$ small. One can formally show that at leading order in $\epsilon$, particles oscillate around the magnetic field lines. 
The drawback of this technique is that there are in fact many drifts appearing at higher order, some due to the geometry of $B$ and one we are specifically concerned with, which is called the electric drift or $E \times B$ drift:
 $$v_E=\frac{E\wedge B}{\vert B \vert^2}$$ 
 where $E$ denotes the electric field. 
 
 Since the electric field is induced by the plasma itself, one can not precisely predict its qualitative behaviour and thus this drift may prevent us from getting a good confinement property : if we wait long enough, particles may stop to perfectly turn around the torus and start drifting toward the edge of the tokamak. In order to overcome the effects of the electric drift, the idea is basically to take advantage of the other drifts due to the geometry of $B$.

In the present paper, we make the assumption that the ions of the plasma are at thermodynamic equilibrium and we describe the distribution of electrons by a kinetic equation. For the sake of simplicity, we restrict to the $2D$ problem in the plane orthogonal to $B$, in order to understand the behaviour of the particles in the slice. We take a magnetic field given by $$B = \frac{\mathcal{B}}{\epsilon} e_\varphi$$
with $\epsilon >0$ a small parameter and $\mathcal{B}$ to be fixed later.
We consider the Finite Larmor Radius scaling (see \cite{FS2} for a reference in the mathematical literature) which consists in considering a characteristic spatial length with the same order as the Larmor radius (which is of order $\epsilon$). This scaling allows for a better description of the orthogonal motion and is expected to make the electric drift appear in the limit $\epsilon\rightarrow 0$.
The density $f_\epsilon(t,x,v)$ (with $t>0, x \in \mathbb{T}^2, v \in \mathbb{R}^2$) of the electrons is then given by the following dimensionless Vlasov Poisson system :
\begin{equation}
  \left\{
 \begin{array}{ll}
\partial_t f_\epsilon + \frac{v}{\epsilon}.\nabla_x f_\epsilon + (E_\epsilon+ v^\perp \frac{\mathcal{B}}{\epsilon}).\nabla_v f_\epsilon = 0 \\
  f_{\epsilon,\vert t=0} = f_0  \\ 
  E_\epsilon= -\nabla_x V_\epsilon \\
    \displaystyle{-\Delta_x V_\epsilon = \int f_\epsilon dv-1}
 \end{array}
\right.
  \end{equation}

We denote $x=\begin{pmatrix}x_1 \\ x_2 \end{pmatrix}$, $v=\begin{pmatrix}v_1 \\ v_2 \end{pmatrix}$ in the local orthogonal basis (see figure \ref{tokamak}). For any vector $A=\begin{pmatrix}A_1 \\ A_2 \end{pmatrix} $, we denote  $A^\perp=\begin{pmatrix}A_2 \\ -A_1 \end{pmatrix}$.

\begin{figure}[h]
\psfrag{0}{$0$}
\psfrag{a}{$x_2$}
\psfrag{b}{$x_1$}
 \includegraphics[scale=0.6]{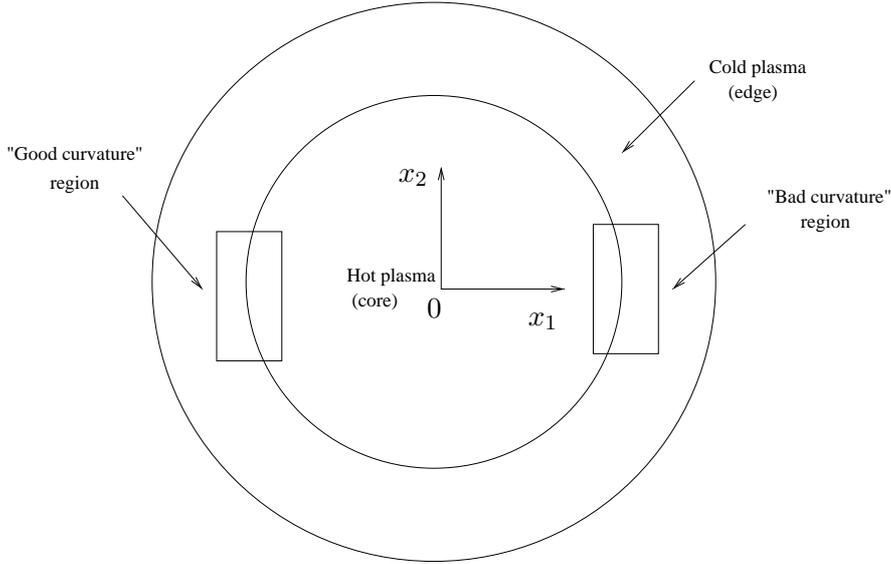}
\caption{A slice of tokamak}
\label{tokamak}
\end{figure}

Following Grandgirard et al. (\cite{Gra}), we consider the explicit formula for $\mathcal{B}$ :
\begin{equation}
\mathcal{B}=\frac{R_0}{R_0 + \epsilon r \cos \theta}= \frac{R_0}{R_0 + \epsilon x_1}
\end{equation}
denoting by $R_0$ the small radius of the torus. (We recall that the characteristic spatial length is of order $\epsilon$)

We consider that $R_0\sim 1$ ; consequently at first order in $\epsilon$ we get:
\begin{equation}
\mathcal{B}=1 - \epsilon x_1
\end{equation}
leading to the following system:
\begin{equation}
\label{larmor}
  \left\{
 \begin{array}{ll}
\partial_t f_\epsilon + \frac{v}{\epsilon}.\nabla_x f_\epsilon + (E_\epsilon+ \frac{v^\perp}{\epsilon} - x_1 v^\perp).\nabla_v f_\epsilon = 0 \\
  f_{\epsilon, \vert t=0} = f_0  \\ 
  E_\epsilon= -\nabla_x V_\epsilon \\
    \displaystyle{-\Delta_x V_\epsilon = \int f_\epsilon dv-1}
 \end{array}
\right.
  \end{equation}
We will see that taking an inhomogeneous intensity for the magnetic field, even at order $1$ in $\epsilon$, leads to a quite different behaviour for the plasma.

Indeed, in the limit $\epsilon\rightarrow 0$, we can derive rigorously  another kinetic system which is qualitatively close to the following one (see sections \ref{gyrolimit} and \ref{mod}):
\begin{equation}
\label{equ}
  \left\{
 \begin{array}{ll}
\partial_t f -\frac{1}{2}\vert v\vert^2\partial_{x_2}f + E^{\perp}.\nabla_x f = 0 \\
  f_{\vert t=0} = f_0  \\ 
  E= -\nabla_x V \\
    \displaystyle{-\Delta_x V = \int f dv-1}
 \end{array}
\right.
  \end{equation}
Observe here that $E^\perp$ corresponds to the electric drift $E \times B$ that we mentioned earlier; the additional drift $v_d=-\frac{1}{2}v^2 e_2$ is due to the inhomogeneity of the magnetic field intensity. The remarkable point is that this drift has a fixed direction; it makes the particles ``fall'' toward the ``bottom'' of the slice. At this point of the modeling, we now have to distinguish between the plasma-core and the plasma edge (see figure 1), the only difference between the two we are concerned with, being that the core is much hotter than the edge. This means from a kinetic point of view that the velocities are much smaller in the edge.

We now divide the slice into two areas: we denote the part $x_1>0$ the \textbf{``bad curvature''} side and the part $x_1<0$ the \textbf{``good curvature''} side: indeed, we expect the plasma in the ``good curvature'' side to be well confined, while the plasma in the ``bad curvature'' region is badly confined. This behaviour can be easily predicted with the following heuristic study in the ``bad curvature'' side:

\begin{center}
 \psfrag{a}{Hot plasma}
 \psfrag{b}{Cold plasma}
 \psfrag{c}{$E$}
 \psfrag{d}{$E$}
 \psfrag{e}{$E^\perp$}
 \psfrag{f}{$E^\perp$}
 \psfrag{g}{$v_d$}
\includegraphics[scale=0.6]{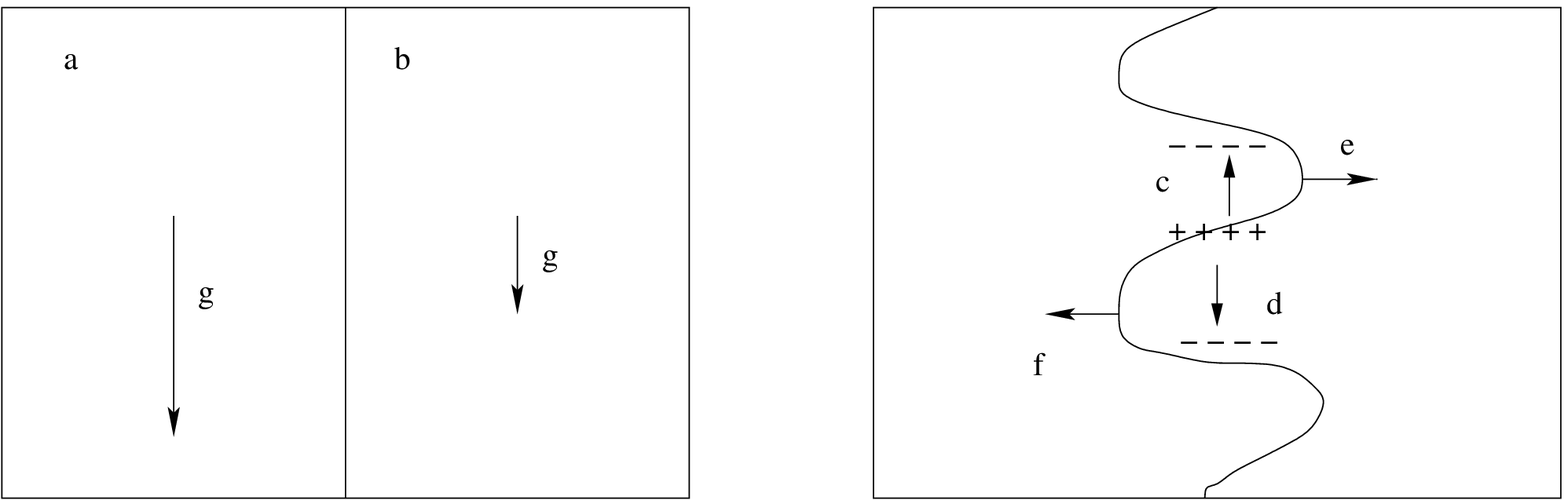}
\end{center}

Particles in the hot plasma drift faster (left figure), so if there is any perturbation (right figure), there appears a separation of charge creating an electric field $E$, which entails a drift $E^\perp$ that accentuates the perturbation: in other words, the equilibrium is unstable.
This discussion is part of the folklore in plasma physics for tokamaks and this instability is recognized to be one of the main sources of disruption for the plasma.

In the other hand one can lead the same qualitative analysis in the ``good curvature'' side and show in this case stability. 


\subsection{Objectives and results of this paper}

In this paper, following the previous heuristic argument, we will particularly focus on system (\ref{model1}), which is a kind of simplified "bi-fluid" version of (\ref{equ}). We consider that the plasma is made of two mixable phases, one being the hot plasma (with constant temperature $T^+$ and density $\rho^+(t,x)$) and the other the cold plasma (with constant  temperature $T^-$ and density $\rho^-(t,x)$) . Of course, hot and cold means that $T^+ >T^-$.
 
\begin{equation}
\label{model1}
  \left\{
 \begin{array}{ll}
\partial_t \rho^+ -T^+\partial_{x_2}\rho^+ + E^\perp.\nabla_x \rho^+ = 0 \\
\partial_t \rho^-  -T^-\partial_{x_2}\rho^+ + E^{\perp}.\nabla_x \rho^- = 0 \\
  E= -\nabla_x V \\
    \displaystyle{-\Delta_x V = \rho^+ + \rho^--1} \\
V=0 \text{  on } x_1=0,L \\
(\rho^+, \rho^-)_{\vert t=0}= (\rho^+_0, \rho^-_0) \text{ with } \int \rho^+_0 + \rho^-_0 = 1
 \end{array}
\right.
  \end{equation}
for $t\geq 0, x \in [0,L]\times {\mathbb{R}}/{L\mathbb{Z}}$ and $L$ is the size of the box.

The temperature $T(t,x)$ of the plasma is given by:
\begin{equation}
T(t,x)= \frac{\rho^+(t,x) T^+ + \rho^-(t,x) T^- }{\rho^+(t,x)  + \rho^-(t,x)  }
\end{equation}
(We refer to Section \ref{mod} for more details.)
 
Unfortunately, we were able to derive this system only formally  from system (\ref{larmor}) and had to make some physical and mathematical approximations. These are precisely explained in Section \ref{mod}. 

We observe that this system shares structural similarities with 2D Euler equations in vorticity form, which describe an incompressible inviscid fluid. 
\[
\text{Plasma} \quad \leftrightarrow \quad \text{Fluid}
\]
\[
\text{density} \quad \rho- 1 \quad \leftrightarrow \quad \text{vorticity}  \quad \omega
\]
\[
\text{(rotated) electric field} \quad  E^\perp = \nabla^\perp \Delta^{-1}(\rho-1)  \quad \leftrightarrow \quad \text{velocity} \quad  u= \nabla^\perp \Delta^{-1}\omega
\]

Hence, our system can be seen somehow as a Euler system with two kinds of vorticities.

Such an analogy between strongly magnetized plasmas and bi-dimensional ideal fluids has been observed  for a long time by physicists (for instance, see \cite{HM}). We mention that the convergence towards $2D$ Euler in strong magnetic fields regimes (but different from the one studied here) was rigorously established by Golse and Saint-Raymond \cite{GSR1} and Brenier \cite{Br}.

Let us now define precisely the stability and instability notions that we will work on until the end of the paper. One should be aware that for such infinite-dimensional dynamical systems, the choice of the norm is particularly important.

\begin{deft}
 Let $\xi$ be a solution to (\ref{model1}).  
This solution is said to be stable with respect to the $X$ norm if for any $\eta>0$, there exists $\delta>0$ such that: for any solution $\rho$ to (\ref{model1}), the initial control $\Vert \rho(0)-\xi(0)\Vert_X\leq \delta$ implies that for any $t\geq 0$, $\Vert \rho(t)-\xi(t)\Vert_X\leq \eta$.

Otherwise, the solution $\xi$ is said to be unstable  with respect to the $X$ norm.
\end{deft}
 Of course, instability will then be interpreted as bad confinement, and stability as good confinement. 

We will investigate stability and instability around the following steady states, modeling the good and bad curvature sides:

\begin{equation}
 \mu^{bad}(x_1)= \left({\mu^{bad,+}=1-\frac{x_1}{L}},{\mu^{bad,-}=\frac{x_1}{L}}\right)
\end{equation}

\begin{equation}
 \mu^{good}(x_1)= \left(\mu^{good,+}=\frac{x_1}{L},\mu^{good,-}=1-\frac{x_1}{L}\right)
\end{equation}
which actually model a linear transition between the hot and the cold plasma. Indeed, for $\mu^{bad}$, the temperature of the plasma is given by:
\[
T(t,x)= T^- \frac{x_1}{L} + T^+\left(1-\frac{x_1}{L}\right)
\]
whereas for  $\mu^{good}$, it is given by:
\[
T(t,x)= T^+ \frac{x_1}{L} + T^-\left(1-\frac{x_1}{L}\right)
\]
Hence the profile of the temperature is a straight line with a slope equal to the so-called temperature gradient $\frac{T^+- T^-}{L}$ . We mention that such linear profiles seem physically relevant in the edge of the tokamak (according to the graphs  in \cite{Green} or \cite{Wes}).

Despite this rather rough model, our predictions will qualitatively correspond to observations made by physicists.

Our first objective will be to confirm the linear scenario exposed in the heuristic study  by exhibiting a growing mode with maximal growth for the linearized operator in the bad-curvature area (but only when the temperature gradient $\frac{T^+- T^-}{L}$ is not too large) and by showing that there is no such growing mode in the good-curvature area. 

Then our aim is to show that nonlinear instability also holds. As system (\ref{model1}) looks a lot like 2D Euler, it is not so surprising that techniques allowing to pass from spectral instability to nonlinear instability for 2D Euler may apply here.
On the topic of stability and instability of ideal plane flows, we mention some recent developments;  let us nevertheless emphasize that this list is by no means exhaustive. In \cite{Gre},  Grenier proved instability  in the  $L^2$  velocity norm for some shear flows with zero Lyapunov exponent. In \cite{BGS},  Bardos, Guo and Strauss, following a method introduced in \cite{GSS}, proved instability in the $L^2$ vorticity norm around some steady states. It is assumed that the linearized operator has a growth exceeding the Lyapunov exponent of the steady flow. We also mention the paper of Vishik  and Friedlander \cite{VF} where instability in the $L^2$ velocity norm is proved under the same assumptions on the steady states.  The best result available by now is due to Lin \cite{Lin04a}. Under rather general assumptions on the steady states (in particular, there is no assumption on the growth of the linearized operator), he showed nonlinear instability in the $L^2$ vorticity norm and in the same time, that velocity grows exponentially in the $L^2$ norm. In this work, we will obtain similar results to those of Lin.

On the other hand, let us emphasize that in our linear analysis, when the temperature gradient $\frac{T^+- T^-}{L}$  exceeds a threshold, there is no growing mode in the bad curvature side. We will show that the nonlinear equations inherit this linear property.
This rather unexpected stability phenomenon can be interpreted as the so-called \textbf{High Confinement mode} (\textbf{H-mode} for short), by opposition to the "standard" regime, referred to as the Low Confinement mode (L-mode for short).
The H-mode is a high-confinement regime  obtained by heating of the plasma and triggered when the heating power exceeds some threshold. It has been experimentally observed by physicists for a long time: it was discovered in the ASDEX tokamak \cite{Wag}, we also refer to  \cite{Green}, \cite{Ido} and  (\cite{Wes}, Section 4.13).
We can also remark that the H-mode is accompanied with an increase of the gradient of temperature (\cite{Green},  \cite{Wes}). These experimental observations fit very well with our qualitative results.

There exists a huge literature in physics on this particular topic. Nevertheless, despite a huge amount of works, the H-mode is still rather mysterious. Its understanding, especially the mechanism of transition from L-mode to H-mode is crucial for fusion research. To the very best of our knowledge, the H-mode has never been rigorously justified at the nonlinear or even at the linear level, with such a simple model.

It is sometimes believed that the formation of confining transport barriers is due to a sheared $E\times B$ flow. In some sense, our model shares similarities with linear shear flows (the linearized equations are similar); thus our study is not in contradiction with these considerations.

We finally mention that in the physics papers,  the H-mode is most of the time numerically investigated with more complicated models (including more physics, such as the effect of collisions, friction, energy sources), we refer to \cite{Fig} and references therein. Our model can be seen as a two-temperature caricature of the model of \cite{Fig}.
The transition to the H-mode is also numerically investigated in  \cite{Ito}, where the existence of thresholds is shown. Unfortunately, we were not able to find any analytical formulae for those thresholds that we could have compared with ours.

The main results proved in this paper (Corollary \ref{stabi} and Theorem \ref{insta}) are gathered in the following theorem:

\begin{thm}For system (\ref{model1}):
\begin{enumerate}
\item (Nonlinear stability)

 The equilibrium  $\mu^{good}$ is nonlinearly stable with respect to the $L^2$ norm.
 
 If the temperature gradient  $\frac{T^+- T^-}{L}$ satisfies:
 \begin{equation}
 \frac{T^+- T^-}{L} > \frac{1}{\pi^2}
\end{equation}
 
 then the equilibrium  $\mu^{bad}$ is nonlinearly stable with respect to the $L^2$ norm.

\item (Nonlinear instability)

 If the temperature gradient satisfies
 \begin{equation}
 \frac{T^+- T^-}{L} < \frac{4}{5\pi^2}
\end{equation}

There exist constants $\delta_0,\eta_0>0$ such that for any $0<\delta<\delta_0$ and any $s\geq0$ there exists a solution $\rho$ to (\ref{model1}) with $\Vert \rho(0)- \mu^{bad} \Vert_{H^s} \leq \delta$ but such that:
\begin{equation}
\Vert E(t_\delta) \Vert_{L^2} \geq \eta_0
\end{equation}
denoting $E(t_\delta)= \nabla \Delta^{-1} (\rho^+(t_\delta)+\rho^-(t_\delta)-1)$ the electric field at time $t_\delta= O(\vert \log \delta \vert)$.

In particular, the equilibrium  $\mu^{bad}$ is nonlinearly unstable with respect to the $L^2$ norm.

\end{enumerate}
\end{thm}

\subsection{Organization of the paper}

The present paper is organized as follows: section \ref{gyrolimit} is devoted to the study of the limit $\epsilon \rightarrow 0$ for the system (\ref{larmor}). In section \ref{mod} we present the simplified bi-fluid model we study in order to investigate stability and instability for the plasma. Section \ref{linearized}  is dedicated to the study of the linearized system around the steady states $\mu^{good}$ and $\mu^{bad}$ ; in particular we show the existence of dominant growing mode in the ``bad curvature'' region, provided that the gradient of temperatures is not too large. If the temperature gradient exceeds some threshold, then there is linear stability.
In section \ref{stability}, we are concerned with the nonlinear stability property for the ``good curvature'' region and for the ``bad curvature'' region for large enough temperature gradients (referred to as the high confinement mode in plasma physics), which will be achieved by exhibiting remarkable energies around the steady states.
In section \ref{instability},  for small enough temperature gradients we pass from linear spectral instability to nonlinear  instability in the $L^2$ vorticity norm, using a high order approximation method introduced by Grenier. Then we prove that the electric field also grows exponentially in the $L^2$ norm, by using the energy exhibited in the previous section. 

\section{Gyrokinetic derivation of the equations}
\label{gyrolimit}

Following Frénod and Sonnendrücker (\cite{FS2}), we can use two-scale convergence tools in order to derive the gyrokinetic equation we are interested in. We shall not dwell on the rigorous derivation of this system since the justifications in two dimensions are essentially done in \cite{FS2}.

First of all, let us recall precisely the two-scale convergence notions (due to Nguetseng \cite{Ngu} and Allaire \cite{Al}) we will use in this section.
 \begin{deft}
 Let $X$ be a separable Banach space, $X'$ be its topological dual space and $(.,.)$ the duality bracket between $X'$ and $X$.
 For all $\alpha>0$, denote by $\mathcal{C}_{\alpha}(\mathbb{R},X)$ (respectively $L^{q'}_{\alpha}(\mathbb{R} ;X')$) the space of $\alpha$-periodic continuous (respectively $L^{q'}$) functions on $\mathbb{R}$ with values in $X$.
 Let $q\in[1;\infty[$. 
 
 Given a sequence $(u_\epsilon)$ of functions belonging to the space $L^{q'}(0,t ;X')$ and a function $U^0(t,\theta) \in L^{q'}(0,T ; L^{q'}_{\alpha}(\mathbb{R} ;X'))$ we say that
 $$u_\epsilon \text{   2-scale converges to   } U^0$$
 if for any function $\Psi \in L^q(0,T ; \mathcal{C}_{\alpha}(\mathbb{R},X))$ we have:
 \begin{equation}
 \lim_{\epsilon \rightarrow 0} \int_0^T \left(u_\epsilon(t), \Psi\left(t,\frac{t}{\epsilon}\right)dt \right)=\frac{1}{\alpha} \int_0^T \int_0^{\alpha} \left(U^0(t,\tau),\Psi(t,\tau) \right) d\tau dt
 \end{equation}
 \end{deft}

The new variable $\tau$ has to be understood as a ``fast-time variable'' which describes the fast time oscillations. As for weak-star convergence in $L^p$ spaces, one can show that boundedness implies $2-$scale convergence in $L^p$ spaces.

 \begin{thm}\label{2scale}(\cite{Ngu}, \cite{Al})
 
 Given a sequence $(u_\epsilon)$ bounded in $L^{q'}(0,t ;X')$, there exists for all $\alpha>0$ a function $U^0_\alpha \in L^{q'}(0,T ; L^{q'}_{\alpha}(\mathbb{R} ;X')$ such that up to a subsequence, 
  $$u_\epsilon \text{   2-scale converges to   } U^0_\alpha$$
  The profile $U^0_\alpha$ is called the $\alpha$-periodic two scale limit of $u_\epsilon$ and the link between $U^0_\alpha$ and the weak-* limit $u$ of $u_\epsilon$ is given by:
  \begin{equation}
  \frac{1}{\alpha}\int_0^\alpha U^0_\alpha d\tau = u
  \end{equation}
 \end{thm}

For the reader's sake we recall the main arguments and refer to \cite{Al} for the complete proof.

\begin{proof}[Sketch of proof]Let $\alpha>0$.
We can consider $\varphi_{u_\epsilon}$ :
\[
\Psi \in L^q(0,T ; \mathcal{C}_{\alpha}(\mathbb{R},X)) \mapsto \int_0^T u_\epsilon(t) \Psi\left(t,\frac t \epsilon \right) dt
\]
and show that it is a continuous linear form on $L^q(0,T ; \mathcal{C}_{\alpha}(\mathbb{R},X))$, so that it can be identified with a unique $U_\epsilon$ in $L^q(0,T ; \mathcal{C}_{\alpha}(\mathbb{R},X))'$.
Then we can show that $U_\epsilon$ is uniformly bounded in $L^q(0,T ; \mathcal{C}_{\alpha}(\mathbb{R},X))'$ ; thus, since $L^q(0,T ; \mathcal{C}_{\alpha}(\mathbb{R},X))$ is a separable Banach space, $U_\epsilon$ weakly-* converges  up to a subsequence  to some $U$ in $L^q(0,T ; \mathcal{C}_{\alpha}(\mathbb{R},X))'$. Using  Riesz's representation theorem, one can show that it can be identified with some $U^0_\alpha \in L^{q'}(0,T ; L^{q'}_{\alpha}(\mathbb{R} ;X'))$  and that $u_\epsilon$ two-scale converges to $U_0$.

\end{proof}

We can now state the main result of this section:

  \begin{prop} 
 For each $\epsilon$, let $f_\epsilon$ be a global weak solution to (\ref{larmor}) in the sense of Arsenev .
 
Then, up to an extraction, $f_\epsilon$ 2-scale converges to a function $F$:
 \begin{equation}
 F(t,\tau,x,v)=G(t,x+\mathcal{R}(\tau)v,R(\tau)v)
 \end{equation}
 and $G$ satisfies:
\begin{equation}
\label{gyro}
  \left\{
 \begin{array}{ll}
\partial_t G + \left(\frac{1}{2\pi}\int_0^{2\pi} \mathcal{R}(\tau)\mathcal{E}(t,\tau,x+\mathcal{R}(-\tau)v)d\tau +\begin{pmatrix} -v_1(v_2-x_1) \\ v_2(x_1-v_2) -\frac{1}{2}(v_1^2 +  v_2^2)\end{pmatrix}\right) .\nabla_x G \\+ \left(\frac{1}{2\pi}\int_0^{2\pi} R(\tau)\mathcal{E}(t,\tau,x+\mathcal{R}(-\tau)v)d\tau + \begin{pmatrix} v_2(-x_1+v_2) \\ -v_1(-x_1+v_2) \end{pmatrix}\right).\nabla_v G = 0 \\
  G_{\vert t=0} = f_0  \\ 
  E= -\nabla_x V \\
    \displaystyle{-\Delta V = \int G(t,x+\mathcal{R}(\tau)v,R(\tau)v)dv-1}
 \end{array}
\right.
  \end{equation}
denoting by $R$ and $\mathcal{R}$ the linear operators defined by :
  $$R(\tau)= \begin{bmatrix}\cos \tau & -\sin \tau  \\ \sin \tau & \cos \tau  \end{bmatrix}, \quad \mathcal{R}(\tau) = \left(R(-\pi/2) - R(-\pi/2+\tau) \right).$$
\end{prop}

\begin{proof}  We do not wish to develop the very beginning of the proof since it is strictly identical to the one given in \cite{FS2}. 

The first step consists in deriving the so-called constraint equation. To this end, let $\Psi(t,\tau,x,v)$ be a $2\pi$-periodic oscillating test function in $\tau$ and define: 
$$\Psi^\epsilon \equiv \Psi\left(t, \frac{t}{\epsilon}, x,v\right)$$ 
Then we can write the weak formulation of the Vlasov equation against $\Psi^\epsilon$ and pass to the two-scale limit. We find that the two-scale limit of $f_\epsilon(t,x,v)$, denoted by $F(t, \tau, x,v)$, satisfies the following equation:
\begin{equation}
\partial_\tau F + v_\perp.\nabla_x F + v\wedge e_z.\nabla_v F_\alpha=0,
\end{equation}
As a consequence, $F$ is constant along the characteristics meaning that there exists a profile $G$ with:
\begin{equation}
 F(t,\tau,x,v)= G(t,x+\mathcal{R}(\tau)v,R(\tau)v)
\end{equation}
where $R$ and $\mathcal{R}$ are defined in the proposition.

The next step is to determine the profile $G$. We introduce the filtered function $g_\epsilon$:
\begin{equation}
g_\epsilon(t,x,v) =f_\epsilon(t,x+\mathcal{R}(-t/\epsilon)v,R(-t/\epsilon)v)
\end{equation}
which represents the number density from which we have removed the essential oscillations. Notice that this function is chosen so that $g_\epsilon$ two-scale converges, as well as weakly-* converges to $G$.

We easily obtain the equation satisfied by $g_\epsilon$:
   \begin{equation}
\label{main}
\begin{split}
  \partial_t g_\epsilon +&  \mathcal{R}(t/\epsilon)E_\epsilon(t, x+\mathcal{R}(-t/\epsilon)v).\nabla_x g_\epsilon \\
  +&R(t/\epsilon)E_\epsilon(t, x+\mathcal{R}(-t/\epsilon)v).\nabla_v g_\epsilon \\
-& \mathcal{R}(t/\epsilon)\left((x+\mathcal{R}(-t/\epsilon)v)_1. (R(-t/\epsilon)v)^\perp\right).\nabla_x g_\epsilon \\
-& {R}(t/\epsilon)\left((x+\mathcal{R}(-t/\epsilon)v)_1. (R(-t/\epsilon)v)^\perp\right).\nabla_v g_\epsilon =0
\end{split}
 \end{equation}

We now pass to the limit in the sense of distributions. We can prove that the following convergence holds for the nonlinear terms (using elliptic regularity for the electric field to gain some compactness):
\begin{eqnarray}
  \mathcal{R}(t/\epsilon)E_\epsilon(t, x+\mathcal{R}(-t/\epsilon)v).\nabla_x g_\epsilon \rightharpoonup \frac{1}{2\pi}\int_0^{2\pi} \mathcal{R}(\tau)\mathcal{E}(t,\tau,x+\mathcal{R}(-\tau)v)d\tau.\nabla_x G\\
R(t/\epsilon)E_\epsilon(t, x+\mathcal{R}(-t/\epsilon)v).\nabla_v g_\epsilon \rightharpoonup \frac{1}{2\pi}\int_0^{2\pi} R(\tau)\mathcal{E}(t,\tau,x+\mathcal{R}(-\tau)v)d\tau.\nabla_v G
\end{eqnarray}

Likewise, we have the following convergences for the last two terms (here there is basically nothing to justify since these are linear quantities): 

\begin{equation}
\begin{split}
- \mathcal{R}(t/\epsilon)&\left((x+\mathcal{R}(-t/\epsilon)v)_1 . (R(-t/\epsilon)v)^\perp\right).\nabla_x g_\epsilon \\ \rightharpoonup& -\frac{1}{2\pi}\int_0^{2\pi} \mathcal{R}(\tau)\big((x+\mathcal{R}(-\tau)v)_1. (R(-\tau)v)^\perp\big)d\tau.\nabla_x G
\end{split}
\end{equation}

\begin{equation}
\begin{split}
- {R}(t/\epsilon)&\left((x+\mathcal{R}(-t/\epsilon)v)_1. (R(-t/\epsilon)v)^\perp\right).\nabla_v g_\epsilon \\ \rightharpoonup& -\frac{1}{2\pi}\int_0^{2\pi} R(\tau)\big((x+\mathcal{R}(-\tau)v)_1. (R(-\tau)v)^\perp\big)d\tau.\nabla_v G
\end{split}
\end{equation}

We then compute the following quantities: 
\begin{equation}
 -\frac{1}{2\pi}\int_0^{2\pi} R(\tau)\big((x+\mathcal{R}(-\tau)v)_1\times (R(-\tau)v)^\perp\big)d\tau=\begin{pmatrix} v_2(-x_1+v_2)\\-v_1(-x_1+v_2) \end{pmatrix}
\end{equation}
\begin{equation}
 -\frac{1}{2\pi}\int_0^{2\pi} \mathcal{R}(\tau)\big((x+\mathcal{R}(-\tau)v)_1\times (R(-\tau)v)^\perp\big)d\tau=\begin{pmatrix} -v_1(-x_1+v_2)\\v_2 (x_1-v_2)- \frac{1}{2}(v_1^2+ v_2^2) \end{pmatrix}
\end{equation}

This concludes the proof.

\end{proof}

\subsection*{Qualitative interpretation of the gyrokinetic system} 

The influence of the variations of $\mathcal{B}$ is given by the drift/acceleration terms:
$$\begin{pmatrix} -v_1(v_2-x_1) \\ v_2(x_1-v_2)-\frac{1}{2}(v_1^2 +  v_2^2)\end{pmatrix} .\nabla_x G + \begin{pmatrix} v_2(-x_1+v_2) \\ -v_1(-x_1+v_2) \end{pmatrix}.\nabla_v G,
$$

Let us imagine that there is no electric field in the asymptotic equation (\ref{gyro}). Then, the characteristics are given by the following ODEs:
\begin{equation}
   \left\{
 \begin{array}{ll}
  \frac{dx}{dt}=\begin{pmatrix} -v_1(v_2-x_1) \\ v_2(x_1-v_2) -\frac{1}{2}(v_1^2 +  v_2^2)\end{pmatrix} \\
\frac{dv}{dt}=\begin{pmatrix} v_2(-x_1+v_2) \\ -v_1(-x_1+v_2) \end{pmatrix}
 \end{array}
\right.
\end{equation}

At first sight, this dynamical system seems a bit complicated with some unpleasant quadratic terms. Actually, this system has some nice invariants.

First, notice that 
$$\frac{d}{dt}(x_1-v_2)=0.$$ 
This means that $x_1=v_2 + C_1$ (with $C_1=x_1(0)-v_2(0)$). The equation for $v$ can now be written in the simple form:
\begin{equation}
\frac{dv}{dt}=\begin{pmatrix} -C_1 v_2 \\ C_1 v_1 \end{pmatrix}
\end{equation}
The velocity is thus periodic (and we could compute it easily). Notice also that $$\frac{d}{dt}(v_1^2 + v_2^2)=0,$$
so that $v_1^2 + v_2^2=C_2$ (with $C_2=v_1^2(0) + v_2^2(0)$).

We get as well a periodic motion for $x_1$ (since $x_1=v_2 + C_1$).
Finally we notice
 for $x_2$:
\begin{equation}
 \frac{d}{dt}(x_2+v_1)= -\frac{1}{2}(v_1^2 +  v_2^2)= - \frac{1}{2}C_2
\end{equation}

The motion along the $e_2$ direction is hence a sum of a periodic motion plus a fall which only depends on the initial velocity of the particles (and not on their position).
Such a drift of the particles ``in the bottom'' of the tokamak and depending only on the square of their velocity is predicted by physicists and is often referred to as the $\nabla \mathcal{B}$ drift (\cite{Wes}, Section 2.6). To support our discussion we give some graphs of the characteristic curves (figures \ref{x2} and \ref{x1x2}).

\begin{figure}[h]
\begin{minipage}{0.4\textwidth}
\includegraphics[scale=0.5]{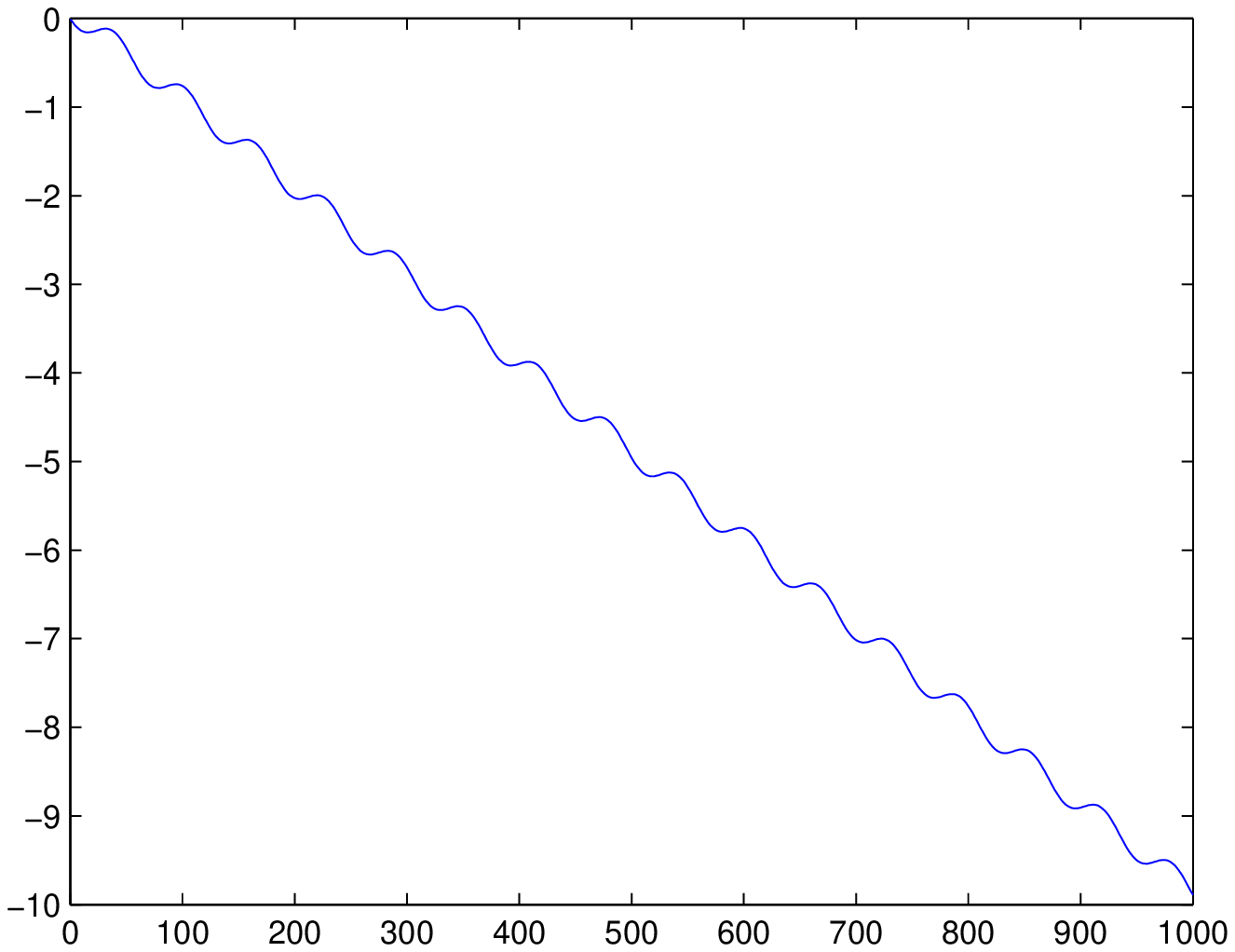}
\caption{$\nabla \mathcal{B}$ drift (in the x-axis: time and in the y-axis: $x_2$)}
\label{x2}
\end{minipage}
\hspace{0.1\textwidth}
\begin{minipage}{0.5\textwidth}
\includegraphics[scale=0.3]{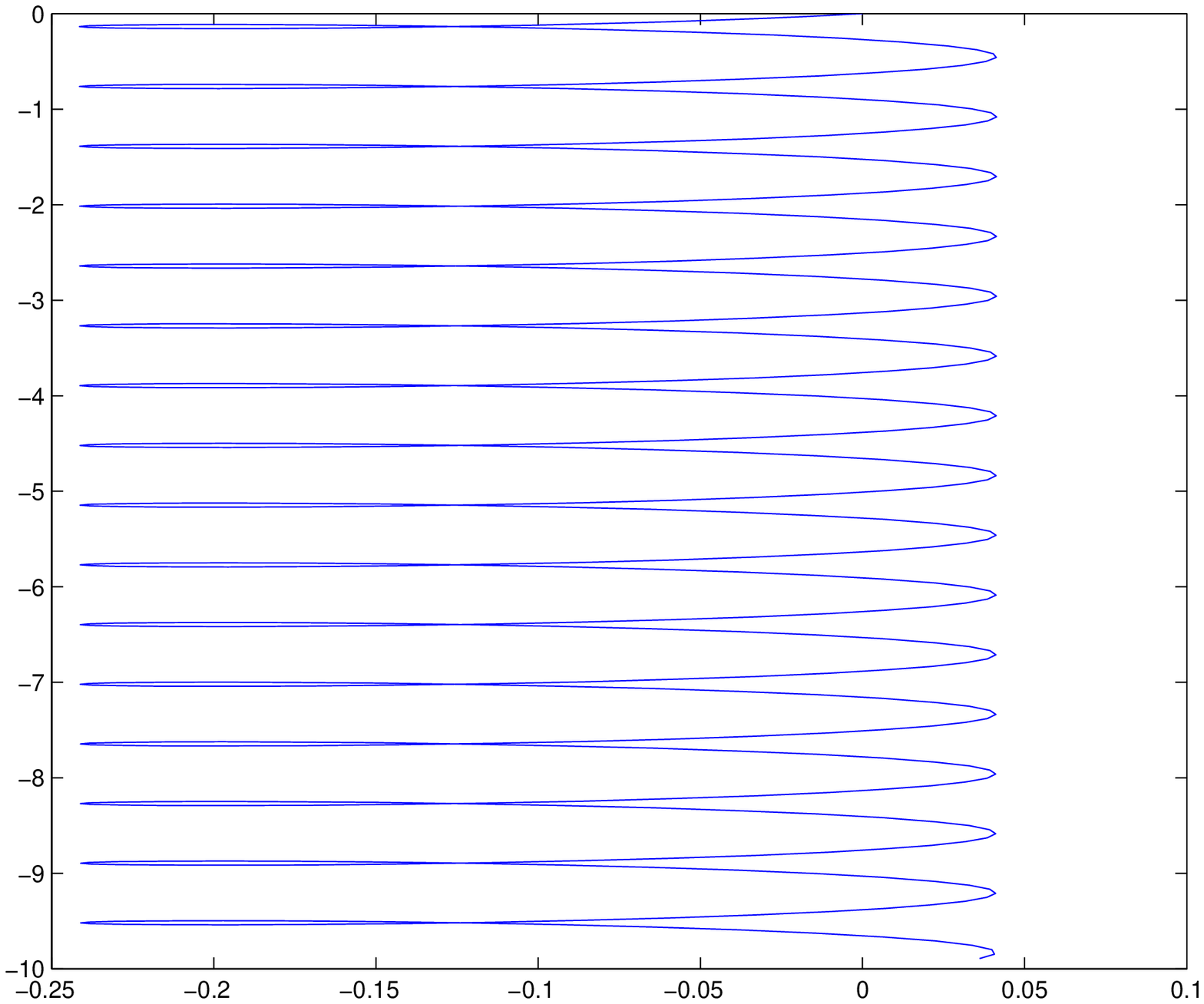}
\caption{Motion of a particle starting at (0,0) in the slice of the tokamak (in the x-axis: $x_1$ and in the y-axis: $x_2$)}
\label{x1x2}
\end{minipage}
\end{figure}

Likewise, Frénod and Sonnendrücker introduced in \cite{FS2} the new variables: 
$$x_c=x-v^\perp=\begin{pmatrix}x_1-v_2 \\ x_2 +v_1\end{pmatrix}$$ 
the so-called guiding center variable and
$$w=-v^\perp$$ 
the so-called Larmor radius variable. With these, they showed that the terms in (\ref{gyro}) involving the electric field were qualitatively close to the drift 
$$ \frac{1}{2\pi}\int_0^{2\pi} \mathcal{E}^\perp (t,\tau,x_c+ R(\tau)w)d\tau.\nabla_{x_c}$$
which corresponds to the gyroaveraged electric drift (\cite{Wes}, Section 2.11).

\section{The simplified mathematical model}
\label{mod}

To be completely rigorous, system (\ref{gyro}) is the system we would have to study in order to investigate good or bad confinement.
Nevertheless, at least at first sight, its algebraic structure seems to be too complicated. Morevore, it is not so clear how to choose a steady state describing the physical situation we want to study.

Consequently we will make several approximations (some of them being quite rough) on (\ref{gyro}) in order to get a more tractable model.

\subsection{A drift-kinetic system}

A first step is to obtain a simplified kinetic system, whose dynamics is close to system (\ref{gyro}).
We therefore consider the following drift-kinetic equation, which is actually a classical physical model (\cite{Wes}, Section 2.11). It is commonly used for numerical simulations (see for instance the GYSELA code \cite{Gra}):
\begin{equation}
\label{pre}
  \left\{
 \begin{array}{ll}
\partial_t f -\frac{1}{2}\vert v\vert^2\partial_{x_{2}}f + E^{\perp}.\nabla_{x} f = 0 \\
  E= -\nabla_{x} V \\
    \displaystyle{-\Delta_{x} V = \int f dv-1}
 \end{array}
\right.
  \end{equation}
 
 This system can be heuristically derived from Newton equations with some elementary physical considerations (see for instance \cite{Wes}, Section 2.6); unfortunately we were not able to derive it rigorously from (\ref{larmor}) or (\ref{gyro}).
 Nevertheless, considering the qualitative study of last paragraph, this seems to be a reasonable model,  if we make the following approximations:
\begin{itemize}
\item We neglect the oscillations in time, which amounts to get rid of the explicit dependance on the fast time variable $\tau$ for the electric field. This can be justified if we consider well-prepared initial data: we refer to the work of Bostan \cite{Bos}.
\item We neglect the gyroaverage operators:
$$ \frac{1}{2\pi}\int_0^{2\pi} E^\perp (t,x_c+ R(\tau)w)d\tau \rightarrow E^\perp(t,x_c)$$
which is reasonable if we consider that the variation of the electric field across a Larmor radius is negligeable.
\end{itemize}

\subsection{The bi-temperature drift-fluid system}

In order to get a simplified fluid model, we assume that the plasma is made of two phases, one being the cold plasma (with low velocities, low temperature $T^-$ and density $\rho^-$) and the other the hot plasma (with large velocities, large temperature $T^+$ and density $\rho^+$). Of course, we take $T^-<T^+$.

Hence, we assume that the solution to (\ref{pre}) takes the form:

\begin{equation}
{f}(t,x,v)={\rho^+(t,x)}\nu^+(v)+{\rho^-(t,x)}\nu^-(v)
\end{equation}
where $\rho^+$ (resp. $\rho^-$) is a positive density such that the total mass is equal to $1$, that is:
$$
\int (\rho^++\rho^-) dx =1.
$$
Furthermore, $\nu^+$ and $\nu^-$ are measures defined by:
 $$\nu^+ = \frac{1}{2\pi \sqrt{2T^+}}\mathbbm{1}_{\vert v \vert = \sqrt{2T^+}}$$ 
$$\nu^- = \frac{1}{2\pi\sqrt{2T^-}}\mathbbm{1}_{\vert v \vert = \sqrt{2T^-}}$$

Considering that transverse particle velocities are isotropically distributed is physically relevant for such magnetized plasmas, as indicated in \cite{Sul}.

 We observe that $\int d\nu^\pm =1$ and $\int v d\nu^\pm=0$. Thus, the charge and current densities are given by:
 \[
 \rho(t,x) := \int f dv = \rho^+(t,x) + \rho^-(t,x)
 \]
 and
 \[
 u(t,x):= \frac{\int f v dv}{\int f dv}=0.
 \]

 In addition, we have:
 $$T^+=\frac 1 2\int \vert v\vert^2 d\nu^+(w) $$
$$T^-=\frac 1 2\int \vert v\vert^2 d\nu^-(w) $$

The kinetic temperature $T(t,x)$ of the plasma is then given by:
\begin{equation}
\begin{split}
T(t,x):=& \frac 1 2 \frac{\int f (v-u(t,x))^2 dv}{\rho(t,x)} \\
=&\frac{\rho^+(t,x) T^+ + \rho^-(t,x) T^- }{\rho^+(t,x)  + \rho^-(t,x)  }
\end{split}
\end{equation}

We moreover assume we can decouple the transport equations satisfied by $\rho^+$ and $\rho^-$.  We get in the end the macroscopic system: 

\begin{equation}
\label{model}
  \left\{
 \begin{array}{ll}
 \partial_t \rho^+ -  T^+ \partial_{x_{2}} \rho^+ + E^\perp.\nabla_{x} \rho^+ =0 \\
  \partial_t \rho^- -  T^- \partial_{x_{2}} \rho^+ + E^\perp.\nabla_{x} \rho^- =0 \\
  E= -\nabla_{x} V \\
    \displaystyle{-\Delta_{x} V = \rho^++ \rho^--1} \\
    (\rho^+, \rho^-)_{\vert t=0}= (\rho^+_0, \rho^-_0) \text{ with } \int \rho^+_0 + \rho^-_0 = 1
 \end{array}
\right.
  \end{equation}
with $x\in [0,L] \times {\mathbb{R}}/{L\mathbb{Z}}$.

As noticed in the introduction, this systems looks like $2D$ incompressible Euler, but with two kinds of vorticities.

Here, the constant  \underline{$L>0$ stands for the size of the box}. The periodicity with respect to $x_2$ is physically justified if we consider that $L$ is small enough with respect to the size of the tokamak, so that we can decompose it in many identical cells of size $L$ (see Figure 1).

We now have to impose some relevant boundary conditions on $x_1=0,L$:
\begin{itemize}
 \item 

For the Poisson equation, we opt for the perfect conductor assumption  on $x_1=0,L$ (which is the ideal case for plasma physics models).
  \begin{equation}
 E^\perp . n  =   \pm E_2= \partial_{{x_c}_2} V = 0 
 \end{equation}
 To this end, we can impose the following Dirichlet boundary condition on $x_1=0,L$:
 \begin{equation}
 V=0
 \end{equation}
  
 From the fluid mechanics point of view, we observe this corresponds to the classical no slip condition. 
\item For the transport equation, we actually do not need any boundary condition. There is indeed no entering or leaving trajectories, since the  linear  ``drift'' operator only entails a motion along the $e_2$ direction, and $E_2=0$ on the boundaries $x_1=0,L$.
\end{itemize}


\par

Following classical works on the Cauchy problem for the $2D$ incompressible Euler system (we refer for instance to the book of Majda and Bertozzi \cite{MajBer}), we get the following global existence and uniqueness result of strong and weak solutions to (\ref{model}): 
\begin{thm}
 Let $\rho_0= (\rho_0^+,\rho_0^-)\in (L^1((0,L)\times\mathbb{R}/L\mathbb{Z}))^2$ with $\rho_0^+,\rho_0^-$ non-negative and $\int (\rho_0^+ + \rho_0^-)dx=1$. 
\begin{enumerate}
\item (Kato, \cite{Kat1}) If $\rho_0$ is $H^s$ (with $s>1$) then there exists a unique classical solution $\rho$ to (\ref{model}) in $\mathcal{C}^0_t([0,\infty[, H^s)\cap \mathcal{C}^1_t([0,\infty[, H^{s-1}) $ with initial data $\rho_0$.

\item (Yudovic, \cite{Yud}) If $\rho_0 \in L^\infty$, then there exists a unique global non-negative weak solution $\rho \in L^\infty_t(L^1\cap L^\infty)$ to (\ref{model}) with initial data $\rho_0$.
\end{enumerate}

\end{thm}

\begin{proof}[Sketch of proof]

\begin{enumerate}

\item The existence of a global strong solution follows from a fixed point argument, as in the classical work on 2D Euler by Kato \cite{Kat1}.
Actually, dealing with two "vorticities" only slightly modifies the main lines; therefore, for the sake of brevity, we only recall here the main arguments of the proof.

 We may consider the map $F: (\xi^+, \xi^-) \mapsto (\rho^+, \rho^-)$ where $ (\rho^+, \rho^-)$ is solution to:
 \begin{equation}
  \left\{
 \begin{array}{ll}
\partial_t \rho^+  + \left(E^\perp-T^+ e_2\right).\nabla_x \rho^+ = 0 \\
\partial_t \rho^- + \left(E^\perp-T^- e_2\right).\nabla_x \rho^- = 0 \\
  E= -\nabla_x V \\
    \displaystyle{-\Delta_x V = \xi^+ + \xi^--1} \\
V=0 \text{  on } x_1=0,L \\
(\rho^+, \rho^-)_{\vert t=0}= (\rho^+_0, \rho^-_0) \text{ with } \int \rho^+_0 + \rho^-_0 = 1
 \end{array}
\right.
 \end{equation}
 which is well-defined if $(\xi^+, \xi^-)$ is smooth enough thanks to the characteristics' method.
As in Kato's proof, for any $T>0$, one can show that $F$ is continuous on some convex compact $S$ of $\left(\mathcal{C}([0,T] \times [0,L]\times \mathbb{R}/L\mathbb{Z})\right)^2$ and that $F(S) \subset S$. The existence of a fixed point is finally a consequence of Schauder's theorem. 
The crucial points are:
\begin{itemize}
\item Establishing some log-lipschitz estimate on the electric field (in this case with a constant involving the $L^\infty$ norms of the two vorticities), which is obtained exactly in the same way as for Euler's equation. 

If $\xi^+,\xi^- \in L^\infty$ and $E$ is solution to the elliptic problem
 \begin{equation}
   \left\{
 \begin{array}{ll}
  E= -\nabla_x V \\
    \displaystyle{-\Delta_x V = \xi^+ + \xi^--1} \\
V=0 \text{  on } x_1=0,L \\
 \end{array}
\right.
 \end{equation}
then there exists $C$ depending on $L$ but independent of $\xi^+, \xi^-$ such that for all $(x,y) \in \left([0,L] \times \mathbb{R}/L\mathbb{Z}\right)$, we have:
\begin{equation}
\vert E(x)-E(y) \vert \leq C \left(\Vert \xi^+ \Vert_{L^\infty}+\Vert \xi^- \Vert_{L^\infty}\right)  \vert x-y\vert \log^+ (\vert x-y\vert)
\end{equation}
where $\log^+(s)=    \left\{
 \begin{array}{ll}
 1-\log s \quad \text{if} \quad s\leq1\\
  0  \quad \text{if} \quad s>1\\
 \end{array}
\right.$

\item The $L^\infty$ norm of each "vorticity" $\rho^+$ and $\rho^-$ is conserved by the transport equations, since $\operatorname{div}_x\left(E^\perp-T^\pm e_2\right)=0$.

\end{itemize}

Finally uniqueness is obtained by an energy argument, exactly as in Kato's proof.

\item  For the bi-dimensional Euler equations in a general domain, the result is due to Yudovic \cite{Yud}. His main arguments can be easily adapted and reproduced in our case, with a crucial use of the two above points.  
The result also follows by a simple adaptation of the alternative proof given in (\cite{MajBer}, Theorem 8.1), which consists in regularizing the initial data, solve the smoothed Cauchy problem and then pass to the weak limit  thanks to a crucial compactness result (\cite{MajBer}, Proposition 8.2).
\end{enumerate}

\end{proof}


\subsection{Modeling of the plasma equilibria}

We now consider the following steady states, in order to model what is happening in the ``good curvature'' or the ``bad curvature'' side of the plasma, the only difference being the relative position between the hot and the cold plasma:
\begin{itemize}
 \item in the ``bad curvature'' region:
\begin{equation}
 \mu^{bad}(x_1)= \left(\underbrace{\mu^{bad,+}=1-\frac{x_1}{L}}_{\text{hot plasma}}, \underbrace{\mu^{bad,-}=\frac{x_1}{L}}_{\text{cold plasma}}\right)
\end{equation}
 
\item in the ``good curvature'' region:
\begin{equation}
 \mu^{good}(x_1)= \left(\mu^{good,+}=\frac{x_1}{L},\mu^{good,-}=1-\frac{x_1}{L}\right)
\end{equation}\end{itemize}

These are steady states of (\ref{model}) and the associated electric field is zero.

We observe here for $\mu^{good}$ the temperature $T(t,x)$ is:
\begin{equation}
T(t,x)= T^+ \frac {x_1}{L}+ T^-\left(1-\frac {x_1}{L}\right)
\end{equation}
Such linear transitions between the cold and hot plasma are the most simple model one can think of. We observe that the slope of the line is equal to $\frac{T^+- T^-}{L}$ which is referred to as the temperature gradient in this paper.

We now investigate stability and instability for $\mu^{bad}$ and $\mu^{good}$.

\section{Linear instability in the ``bad curvature'' region}
\label{linearized}

We first consider the case of the ``bad curvature'' region, for which we expect to obtain instability. The equilibrium writes:
\begin{equation}
 \mu^{bad}(x_1)= \begin{pmatrix}1-\frac{x_1}{L} \\ \frac{x_1}{L} \end{pmatrix} \nonumber
 \end{equation}
The first step before trying to prove any instability property for the nonlinear transport equations consists in investigating the problem of instability for the linearized operator around $\mu^{bad}$. We accordingly consider the following linearized system (for $t>0$, $x \in [0,L]\times \mathbb{R}/L\mathbb{Z}$):

\begin{equation}
\label{lin}
  \left\{
 \begin{array}{ll}
\partial_t \rho^+ -T^+\partial_{x_2}\rho^+ - \frac{E_2}{L} = 0 \\
\partial_t \rho^- -T^-\partial_{x_2}\rho^- + \frac{E_2}{L} = 0 \\
  E= -\nabla_x V \\
    \displaystyle{-\Delta_x V = \rho^+ + \rho^-}\text{ , }V=0 \text{ on } x_1=0,L \\
(\rho^+, \rho^-)_{\vert t=0}= (\rho^+_0, \rho^-_0) \text{ with } \int \rho^+_0 + \rho^-_0 = 0
 \end{array}
\right.
  \end{equation}

\subsection{Looking for unstable eigenfunctions}

We look for special solutions under the form $\rho_k(t,x)=\begin{pmatrix}h_k^+(t)g_k(x) \\ h_k^-(t)g_k(x) \end{pmatrix}$. The equations (\ref{lin}) can be restated as:

\begin{equation}
  \left\{
 \begin{array}{ll}
\partial_t (h_k^+(t)g_k(x)) -T^+\partial_{x_2}(h_k^+(t)g_k(x)) - \frac{E_2}{L} = 0 \\
\partial_t (h_k^-(t)g_k(x)) -T^-\partial_{x_2}(h_k^-(t)g_k(x))+ \frac{E_2}{L} = 0 \\
  E= -\nabla_x V \\
    \displaystyle{-\Delta_x V = (h_k^-(t)+h_k^+(t))g_k(x)}\text{ , }  V=0 \text{ on } x_1=0,L
 \end{array}
\right.
  \end{equation}

We take $g_k$ with the particular form $g_k(x)=\sin{(\frac{k_1}{L} \pi x_1)}e^{i2\pi \frac{k_2}{L} x_2}$ (with $k_1, k_2 \in \mathbb{Z}^*$), so that $g_k$ is an eigenfunction for the laplacian with the considered boundary conditions. It satisfies indeed:
\begin{equation}
 \Delta g_k(x) = - \pi^2\left(\frac{k_1^2}{L^2} + 4\frac{k_2^2}{L^2} \right)g_k(x)
\end{equation}
and also $ g_k=0$ on $x_1=0,L$.

The solution to the Poisson equation is then given by:
$$V_k=\frac{1}{\pi^2\left(\frac{k_1^2}{L^2} + 4\frac{k_2^2}{L^2} \right)}(h_k^-(t)+h_k^+(t))g_k(x)$$
and thus we have 
$$E_2= \frac{-i2 k_2/L}{\pi\left(\frac{k_1^2}{L^2} + 4\frac{k_2^2}{L^2}  \right)}(h_k^-(t)+h_k^+(t))g_k(x);$$
 which leads us to study the following first order ordinary differential equation:
\begin{equation}
 \partial_t \begin{pmatrix} h_k^+(t)\\ h_k^-(t)\end{pmatrix} + 1/L \begin{pmatrix} -2i\pi T^+ k_2+ \frac{i2 k_2}{\pi\left(\frac{k_1^2}{L} + 4\frac{k_2^2}{L} \right)} & \frac{i2 k_2}{\pi\left(\frac{k_1^2}{L} + 4\frac{k_2^2}{L}  \right)}\\ -\frac{i2 k_2}{\pi\left(\frac{k_1^2}{L} + 4\frac{k_2^2}{L}  \right)} & -2i\pi T^- k_2 -\frac{i2 k_2}{\pi\left(\frac{k_1^2}{L} + 4\frac{k_2^2}{L}  \right)} \end{pmatrix}\begin{pmatrix} h_k^+(t) \\ h_k^-(t)\end{pmatrix}=0
\end{equation}

We want to compute the eigenvalues of the matrix; its characteristic polynomial states:
$$X^2 + 2i\pi k_2 \left(T^+ +T^-\right)X -4\pi^2 k^2_2 T^+ T^- - \frac{4k_2^2}{\frac{k_1^2}{L} + 4\frac{k_2^2}{L} }\left(T^+-T^-\right)$$
and its discriminant:
\begin{eqnarray}
\Delta&=&-4\pi^2 k_2^2 (T^+-T^-)^2 + \frac{16k_2^2 L }{{k_1^2} + 4{k_2^2} }\left(T^+-T^-\right)  \nonumber\\
&=&-4\pi^2 k_2^2 (T^+-T^-) \left((T^+-T^- ) - \frac{4L}{\pi^2({k_1^2}+ 4{k_2^2} )} \right)
\end{eqnarray}

We recall that by definition,
\[
T^+ -T^- >0
\]

We can now distinguish between two cases:
\begin{itemize}
 \item First case: 
 
\begin{equation}
\frac{4}{5\pi^2}> \frac{T^+-T^-}{L}
\end{equation} 

 In the case where the gradient of temperature is not too large, then there exist $k_1,k_2$ such that $\Delta>0$. We consequently obtain two complex roots, one of which has a stricly negative real part equal to $-\frac{\sqrt{\Delta}}{2}$. In other words, this shows the existence of an unstable mode.
 \item Second case: 
 
\begin{equation}
 \frac{4}{5\pi^2}\leq  \frac{T^+-T^-}{L}
\end{equation} 
 
 In the opposite case, we always have $\Delta\leq 0$ and consequently, we are not able to find a growing mode ! 
 
This phenomenon may at first sight look like a mathematical artifact due to the periodicity constraint in the $x_2$ direction. Nevertheless as explained in the introduction, the existence of such a threshold is well known in plasma physics, beyond which one can expect tremendous confinement properties. In very rough terms: heating brings stability. The stable mode is referred to as the H-mode, by opposition to the L-mode.
\end{itemize}

\begin{rque}
In the ``good curvature'' region, that is around $\mu_{good}$, the linearized system states:

\begin{equation}
\label{linear}
  \left\{
 \begin{array}{ll}
\partial_t \rho^+ -T^+\partial_{x_2}\rho^+ + \frac{E_2}{L} = 0 \\
\partial_t \rho^- -T^-\partial_{x_2}\rho^-  - \frac{E_2}{L} = 0 \\
  E= -\nabla_x V \\
    \displaystyle{-\Delta_x V = \rho^+ + \rho^-}\text{ , }  V=0 \text{ on } x_1=0,L
 \end{array}
\right.
  \end{equation}

With the same method, we obtain the following ordinary differential equation: 
\begin{equation}
 \partial_t \begin{pmatrix} h_k^+(t)\\ h_k^-(t)\end{pmatrix} + 1/L\begin{pmatrix} -2i\pi T^+ k_2- \frac{i2k_2}{\pi\left(\frac{k_1^2}{L} + 4\frac{k_2^2}{L}  \right)} & -\frac{i2 k_2}{\pi\left(\frac{k_1^2}{L} + 4\frac{k_2^2}{L}  \right)}\\ \frac{i2 k_2/L}{\pi\left(\frac{k_1^2}{L} + 4\frac{k_2^2}{L}  \right)} &-2i\pi T^- + k_2 \frac{i2 k_2}{\pi\left(\frac{k_1^2}{L} + 4\frac{k_2^2}{L}  \right)} \end{pmatrix}\begin{pmatrix} h_k^+(t) \\ h_k^-(t)\end{pmatrix}=0
\end{equation}
We consequently have to look for the roots to the polynomial:
$$X^2 +2 i\pi k_2 v^2X -4\pi^2 T^+ T^- k_2^2 +  \frac{4k_2^2}{\frac{k_1^2}{L}  + 4\frac{k_2^2}{L} }(T^+-T^-)$$
In this case, one can check as before that the discriminant is always stricly negative, so that the roots always have a vanishing real part. As a result, we do not find any unstable mode by this method. Note that we only consider this fact as a good and encouraging indication for stability around $\mu^{good}$. Actually, we will never use it when proving nonlinear stability in section \ref{stability}.
\end{rque}

\begin{rque}
We finally mention that the quest for unstable modes seems more difficult in the kinetic case, since one has to deal with the continuous velocity space. A very famous criterion in the Vlasov-Poisson case was given by Penrose \cite{Pen} and rigorously studied later on by Guo and Strauss \cite{GS}.
\end{rque}

\subsection{On the spectrum of the linearized operator around $\mu^{bad}$ on $H^s([0,L]\times \mathbb{R} / \mathbb{Z})$, $ s\geq 0$}
\label{spec}

We assume here the existence of unstable modes for the linearized operator around $\mu^{bad}$, that is in the situation where we have:

\[
\frac{4}{5\pi^2}> \frac{T^+-T^-}{L}
\]

The main tool we use now is a variant of a classical theorem by Weyl, stated for instance in the paper of Guo and Strauss \cite{GS} and proved by Vidav in \cite{Vid}. Basically, it gives informations on the spectrum of some compact perturbation of a linear operator which has no spectrum in the half-plane $\left\{\operatorname{Re} z >0 \right\}$. It entails the existence of a dominant unstable eigenvalue provided the existence of at least one growing mode.
\begin{thm}[Weyl]
\label{Weyl}
 Let Y be a Banach space and $A$ be a linear operator that generates a strongly continuous semigroup on $Y$ such that $\Vert e^{-tA}\Vert \leq M$ for all $t\geq 0$. Let $K$ be a compact operator from $Y$ to $Y$. Then $(A+K)$ generates a strongly continuous semigroup $ e^{-t(A+K)}$ and  $\sigma(-A-K)$ consists of a finite number of eigenvalues of finite multiplicity in $\{\operatorname{Re} \lambda>\delta \}$ for every $\delta>0$. These eigenvalues can be labeled by:
$$\operatorname{Re} \lambda_1\geq \operatorname{Re} \lambda_2\geq ... \operatorname{Re} \lambda_N\geq \delta$$
Furthermore, for any $\gamma>\operatorname{Re} \lambda_1$, there exists some constant $C_\gamma$ such that
\begin{equation}
 \Vert e^{-t(A+K)}\Vert_{Y\rightarrow Y} \leq C_\gamma e^{t \gamma}
\end{equation}
\end{thm}

\begin{coro}
\label{spectral}
Let $s\geq0$ and 
$$Y=\{y_1,y_2 \in H^s([0,L]\times \mathbb{R} / \mathbb{Z})^2, \int (y_1+y_2)dx=0\}.$$
Let $M$ be the linear operator defined by:
\begin{equation} g=\begin{pmatrix} g^+ \\ g^-\end{pmatrix} \in Y \mapsto M\begin{pmatrix} g^+ \\ g^-\end{pmatrix} =\begin{pmatrix} -T^+\partial_{x_2}g^+- \frac{E_2}{L}  \\-T^-\partial_{x_2}g^- + \frac{E_2}{L} \end{pmatrix}\end{equation}
with $E_2 =-\partial_{x_2}V $, $-\Delta V =(g^+ +g^-)$ and with $V=0$ on  $x_1=0,L$.

Then there exists an eigenvalue $\lambda$ with a non-vanishing and maximal real part associated to a $C^\infty$ eigenvector. Furthermore for any $\gamma > \operatorname{Re} \lambda$, there is a constant $C(\gamma,s)$ such that for all $t\geq0$:
\begin{equation}
 \Vert e^{-tM}\Vert_{H^s\rightarrow H^s} \leq C(\gamma,s) e^{t \gamma}
\end{equation}
\end{coro}

\begin{proof}
The linear operator $A$, defined by 
\[
A: g\mapsto \begin{pmatrix}-T^+\partial_{x_2} g^+ \\-T^-\partial_{x_2} g^-\end{pmatrix}
\]
 is clearly an isometry on $Y$ (indeed we know how to explicitly solve the semi-group). The operator $K$ is defined by 
\[
K: g \mapsto \begin{pmatrix} -\frac{E_2}{L}  \\ \frac{E_2}{L}  \end{pmatrix}
\]
 
 This operator is compact on $Y$ thanks to standard elliptic estimates.

Moreover, we have shown in the last paragraph the existence of an unstable eigenfunction for the linearized operator that belongs to any $H^s$. We can therefore apply Weyl's theorem which gives the existence of an eigenfunction associated to an eigenvalue with a non-vanishing and maximal real part.
At last, the estimate in the corollary follows directly from the estimate given in Weyl's theorem.

\end{proof}

In the following, we denote for any $h \in L^2$ with $\int h dx=0$, $\Delta^{-1} h$ the unique solution u in $H^1$ to the problem:
$$
  \left\{
 \begin{array}{ll}
-\Delta u = h \\
 u=0 \text{ on } x_1=0,L \\
 \end{array}
\right.
$$

We give now a lemma which tells us that any eigenvector associated to a non vanishing eigenvalue for the linearized operator any $L^2$ actually has $\mathcal{C}^\infty$ regularity.

\begin{lem}
\label{reg}
Let $\rho=(\rho^+, \rho^-)\in (L^2)^2$ with $\int (\rho^+ + \rho^-)dx=0$ and $\lambda\neq 0$ such that:
\begin{eqnarray*}
 -T^+\partial_{x_2} \rho^+ - \frac{1}{L}\partial_{x_2}\Delta^{-1}(\rho^++\rho^-)&=& \lambda \rho^+ \\ -T^-\partial_{x_2} \rho^- -
 \frac{1}{L}\partial_{x_2}\Delta^{-1}(\rho^++\rho^-)&=& \lambda \rho^-
\end{eqnarray*}
then $(\rho^+, \rho^-)\in \mathcal{C}^\infty([0,L]\times\mathbb{R}/\mathbb{Z})$
\end{lem}

\begin{proof}
 The principle of the proof is to show by recursion that $\rho \in H^k$, for any $k\in\mathbb{N}^*$.

For $k=1$, we can observe, thanks to elliptic estimates, that $\partial_{x_2}\rho \in L^2$. Indeed, we have the identity: 
\begin{equation}
\label{ident}
 -T^+ \partial_{x_2} \rho^+= \frac{1}{L}\partial_{x_2}\Delta^{-1}(\rho^++\rho^-) + \lambda \rho^+ 
\end{equation}
Hence, $\partial_{x_2} \rho^+ \in L^2$. Likewise, $\partial_{x_2} \rho^- \in L^2$.

We can apply the differential operator $\partial_{x_1}$ to the equation satisfied by $\rho^+$, which entails: 
\begin{eqnarray*}
 -T^+ \partial_{x_1}\partial_{x_2} \rho^+ - \frac{1}{L}\partial_{x_1}\partial_{x_2}\Delta^{-1}(\rho^++\rho^-)&=& \lambda \partial_{x_1}\rho^+ 
\end{eqnarray*}

Then we multiply by $\partial_{x_1}\rho^+$ and integrate with respect to $x$:

\begin{eqnarray*}
\lambda \Vert \partial_{x_1} \rho^+\Vert_{L^2}^2 &=&  \int -T^+ \partial_{x_1}\partial_{x_2} \rho^+ \partial_{x_1}\rho^+ dx -\frac{1}{L}\int \partial_{x_1}\partial_{x_2}\Delta^{-1}(\rho^++\rho^-)\partial_{x_1}\rho^+ dx
\end{eqnarray*}
Thanks to the periodicity with respect to $x_2$, we get: 

$$\int \partial_{x_1}\partial_{x_2} \rho^+ \partial_{x_1}\rho^+ dx=1/2\int \partial_{x_2}\Big( \partial_{x_1}\rho^+\Big)^2dx=0$$

Then using Cauchy-Schwarz inequality:
\begin{eqnarray*}
\lambda \Vert \partial_{x_1} \rho^+\Vert_{L^2}^2 \leq \frac{1}{L} \Vert \partial_{x_1}\partial_{x_2} \Delta^{-1}(\rho^+ + \rho^-)\Vert_{L^2}\Vert \partial_{x_1} \rho^+\Vert_{L^2}
\end{eqnarray*}

As a result we showed that:
\begin{equation}
 \lambda^2\Vert \partial_{x_1} \rho^+\Vert_{L^2}  \leq \frac{1}{L} \Vert \partial_{x_1}\partial_{x_2} \Delta^{-1}(\rho^+ + \rho^-)\Vert_{L^2}
\end{equation}
By standard elliptic estimates the right-hand side is finite since $\rho^+$ et $\rho^-$ belong to $L^2$. As a result we have proved $\rho \in H^1$.

We can then conclude by recursion. Let us assume that $\rho \in H^{k}$, for some $k\in \mathbb{N}^*$; we prove that $\rho \in H^{k+1}$. 

Let $\alpha, \beta \in \mathbb{N}$ such that $\alpha +\beta =k$. We set $\partial_k = \partial^\alpha_{x_1}  \partial^\beta_{x_2} $. Then $\partial_k \rho$ belongs to $L^2$ and satisfies the equation:
\begin{eqnarray*}
 -T^+\partial_{x_2} \partial_k \rho^+ - \frac{1}{L} \partial_k \partial_{x_2}\Delta^{-1}(\rho^++\rho^-)&=& \lambda  \partial_k \rho^+ \\ -T^-\partial_{x_2}  \partial_k \rho^- -
 \frac{1}{L} \partial_k \partial_{x_2}\Delta^{-1}(\rho^++\rho^-)&=& \lambda  \partial_k \rho^-
\end{eqnarray*}

By elliptic regularity, $ \partial_k \partial_{x_2}\Delta^{-1}(\rho^++\rho^-) \in H^{1}$. Thus, we are in the same case as for $L^2 \rightarrow H^1$, which entails that $\rho \in H^{k+1}$.

\end{proof}

\section{On nonlinear stability}
\label{stability}

Let $\rho=\begin{pmatrix}\rho^+ \\ \rho^-\end{pmatrix}$ a solution to the nonlinear transport equation  (\ref{model}), that we recall here:
\begin{equation*}
  \left\{
 \begin{array}{ll}
\partial_t \rho^+ -T^+\partial_{x_2}\rho^+ + E^\perp.\nabla_x \rho^+ = 0 \\
\partial_t \rho^-  -T^-\partial_{x_2}\rho^+ + E^{\perp}.\nabla_x \rho^- = 0 \\
  E= -\nabla_x V \\
    \displaystyle{-\Delta_x V = \rho^+ + \rho^--1} \\
V=0 \text{  on } x_1=0,L \\
(\rho^+, \rho^-)_{\vert t=0}= (\rho^+_0, \rho^-_0) \text{ with } \int \rho^+_0 + \rho^-_0 = 1
 \end{array}
\right.
  \end{equation*} 

We begin with a very simple observation in the limit case $T=T^+=T^-$. In this situation, setting $\tilde\rho=\rho^+ + \rho^-$, $\rho$ satisfies the usual $2D$ Euler equation in vorticity form, with some linear drift term:
\[
\partial_t \tilde\rho + E^\perp.\nabla_x \tilde\rho - T \partial_{x_2} \tilde\rho=0
\]
\[
E^\perp = \nabla^\perp \Delta^{-1}(\tilde\rho-1)
\]
We investigate stability around the steady state $\tilde \mu= \mu^+ + \mu^- = 1$. It is well-known that any $L^p$ norm of the "vorticity" $\tilde \rho -1$ is non-increasing, that is:
\begin{equation}
\Vert \tilde\rho(t)-1 \Vert_{L^p} \leq \Vert \tilde\rho(0)-1 \Vert_{L^p}
\end{equation}
which clearly entails nonlinear $L^p$ stability.

The idea in the general case $T^+ >T^-$ in order to show nonlinear stability is to obtain some similar nice energy estimate.

\begin{thm}
\label{stab}For any inital data $\rho_0 \in L^\infty$ with $\int (\rho^+_0 + \rho^-_0 )dx =1$, the solution $\rho$ to (\ref{model}) satisfies the following statements.

\begin{itemize}
\item Around the "good-curvature" steady state, the following functional is non-increasing:
 \begin{equation}
 \label{E}
 \mathcal{E}(t)=\Vert \rho-\mu^{good} \Vert_{L^2}^2 + \frac{1}{L(T^+-T^-)}\int \vert \nabla V\vert ^2 dx \leq \mathcal{E}(0)
\end{equation}
with  $\Vert \rho-\mu \Vert_{L^2}^2=\Vert \rho^+-\mu^+ \Vert_{L^2}^2+\Vert \rho^--\mu^- \Vert_{L^2}^2$ and $\nabla V = \nabla \Delta^{-1} (\rho^+ + \rho^- -1)$.

\item Around the "bad-curvature" steady state, the following functional is non-increasing:
 \begin{equation}
 \mathcal{F}(t)=\Vert \rho-\mu^{bad} \Vert_{L^2}^2 - \frac{1}{L(T^+-T^-)}\int \vert \nabla V\vert ^2 dx \leq  \mathcal{F}(0)
\end{equation}

\end{itemize}

\end{thm}

\begin{rque}
We observe that in the functionals, the first term corresponds to the enstrophy in fluid mechanics (in a modulated form adapted to our needs). The second term can be interpreted as the kinetic energy of the fluid (whereas from the plasma physics point of view, this is the electric energy).
\end{rque}

As an immediate consequence of this theorem, we obtain $L^2$ stability, in the "good-curvature" side, but also in the "bad-curvature" side, for large enough temperature gradients, like for the linearized equations.

\begin{coro}
\label{stabi}
 The equilibrium  $\mu^{good}$ is nonlinearly stable with respect to the $L^2$ norm.
 
 If the temperature gradient $ \frac{T^+- T^-}{L}$ satisfies
 \begin{equation}
 \frac{T^+- T^-}{L} > \frac{1}{\pi^2}
\end{equation}
 
 then the equilibrium  $\mu^{bad}$ is nonlinearly stable with respect to the $L^2$ norm.
\end{coro}

\begin{proof}[Proof of the corollary] Thanks to the energy identity (\ref{E}) and to the Poisson equation:
\begin{eqnarray}
  \Vert \rho-\mu^{good} \Vert_{L^2}^2 &\leq& \Vert \rho-\mu^{good} \Vert_{L^2}^2 +  \frac{1}{L(T^+-T^-)}\int \vert \nabla V\vert ^2 dx  \nonumber\\
&\leq& \Vert \rho(0)-\mu^{good} \Vert_{L^2}^2 +  \frac{1}{L(T^+-T^-)}\int \vert \nabla V(0)\vert ^2 dx\nonumber\\
&\leq & \Vert \rho(0)-\mu^{good} \Vert_{L^2}^2 + C \frac{1}{L(T^+-T^-)}\Vert \rho(0)-\mu^{good} \Vert_{L^2}^2
\end{eqnarray}
This means that for any $\eta>0$ there exists $\delta>0$ such that if $\Vert \rho(0)-\mu^{good} \Vert_{L^2}\leq\delta$ then for any $t\geq 0$, $\Vert \rho-\mu^{good} \Vert_{L^2}\leq \eta$. In other words, $\mu^{good}$ is nonlinearly stable for the $L^2$ norm.

Around the bad-curvature steady state, we have shown that the following quantity is non-increasing:
\begin{equation}
  \mathcal{F}(t)=\Vert \rho-\mu^{bad} \Vert_{L^2}^2 - \frac{1}{L(T^+-T^-)}\int \vert \nabla V\vert ^2 dx
\end{equation}
We can easily prove with the help of Fourier variables the Poincaré-like inequality:
\begin{equation}
 \int \vert \nabla V\vert ^2 dx \leq \frac{ L^2 }{\pi^2}\Vert \rho-\mu^{bad}\Vert_{L^2}^2
\end{equation}
So we get:
\begin{eqnarray*}
  \Vert \rho-\mu^{bad}\Vert_{L^2}^2 &\leq& \mathcal{F}(0) + \frac{1}{L(T^+-T^-)}\int \vert \nabla V\vert ^2 dx \\
&\leq& \mathcal{F}(0) + \frac{1}{L(T^+-T^-)} \frac{L^2}{\pi^2}\Vert \rho-\mu^{bad} \Vert_{L^2}^2
\end{eqnarray*}
Hence,
\begin{equation}
 \left(1-\frac{L }{\pi^2 (T^+-T^- )}\right)\Vert \rho-\mu^{bad} \Vert_{L^2}^2 \leq \Vert \rho(0)-\mu^{bad} \Vert_{L^2}^2
\end{equation}

As a consequence, there is $L^2$ nonlinear stability in the bad curvature side, provided that 
\[
 \frac{T^+- T^-}{L} > \frac{1}{\pi^2}.
\]
Otherwise, we can not deduce anything.

\end{proof}

\begin{rque}
Note also that in the linear discussion, there was stability provided that
\begin{equation}
 \frac{T^+- T^-}{L}  \geq \frac{4}{5\pi^2}.
\end{equation}

We do not know what happens for  $ \frac{T^+- T^-}{L} \in [\frac{4}{5\pi^2},\frac{1}{\pi^2}]$ in the nonlinear case. Maybe one could expect to observe some bifurcation phenomenon.

\end{rque}

\begin{rque}
Actually the decrease of $\mathcal{E}(t)$ tells us a little more than just $L^2$ stability.
Indeed, there exists $C>0$ such that for any $\delta>0$, if $ \Vert \rho(0)-\mu^{good}\Vert_{L^2}^2 \leq \delta$, then for any $t>0$,
\[
\Vert \rho(t)-\mu^{good} \Vert_{L^2}^2 \leq \left(1+C\frac{1}{L(T^+-T^-)} \right)\delta.
\]
This means in particular that for large values of  $ T^+- T^-$, better confinement is obtained, which is qualitatively in agreement with experimental observations (\cite{Green}).
\end{rque}

We now prove Theorem \ref{stab}. We first give a technical lemma in  which will help for the proof.
\begin{lem}
\label{growth}
For $\mu= \mu^{good}$ or $\mu^{bad}$, we have for any $t>0$,
\begin{equation}
\label{identity}
\int E_2 \left(\rho^+-\mu^+\right)dx=-\int E_2 \left(\rho^--\mu^-\right)dx.
\end{equation}
\end{lem}

\begin{proof}
In order to prove this identity, one can simply compute:
\begin{eqnarray*}
 \int E_2 \left((\rho^--\mu^-) -(\rho^+-\mu^+) \right) dx &=&  \int E_2 \left(\rho^++\rho^--\mu^+-\mu^--2(\rho^+-\mu^+)\right)dx \\
&=& \int E_2 \left(\Delta V-2(\rho^+-\mu^+)\right)dx\\
&=& -2\int E_2 \left(\rho^+-\mu^+\right)dx
\end{eqnarray*}
Indeed, thanks to periodicity with respect to  $x_2$ and since $\partial_{x_2} V=0$ on $x_1=0,L$, we get: 
\begin{eqnarray*}
\int \partial_{x_2}V  \Delta Vdx &=& - \int \partial_{x_2} \nabla V.\nabla  Vdx  + \underbrace{\int \operatorname{div}(\partial_{x_2}V\nabla V )dx}_{=0}\\
&=& \int \partial_{x_2}\left( \frac{\vert\nabla V\vert^2}{2}\right)dx \, = \,0 
\end{eqnarray*}

If we make the same computation, by symmetry, we can also observe that:
 $$ \int E_2 \left((\rho^--\mu^-) -(\rho^+-\mu^+) \right) dx=2\int E_2 \left(\rho^--\mu^-\right)dx$$

\end{proof}

\begin{proof}[Proof of Theorem \ref{stab}]

We will only focus on the proof of the conservation of $\mathcal{E}(t)$, the proof being very similar for $\mathcal{F}(t)$. For the sake of readibility, we write $\mu$ instead of $\mu^{good}$ until the end of the proof.

We observe that the equation satisfied by $(\rho-\mu)$ reads in this case:
\begin{equation}
 \partial_t (\rho-\mu) - \begin{pmatrix} T^+ \partial_{x_2} (\rho^+-\mu^+) \\ T^-\partial_{x_2} (\rho^- -\mu^-)\end{pmatrix} + E^\perp.\nabla_x (\rho-\mu) = \begin{pmatrix}-\frac{E_2}{L} \\ \frac{E_2}{L} \end{pmatrix}
\end{equation}
and $E=-\nabla_x V$ with $-\Delta V = \rho^++\rho^--\mu^+-\mu^-$ and with Dirichlet conditions on the boundaries $x_1=0$ and $x_1=1$.

Taking the scalar product with $(\rho-\mu)$ in the transport equation and integrating with respect to $x$ entails:
\begin{equation}
 \frac{d}{dt} \Vert \rho-\mu \Vert_{L^2}^2 = \int -\frac{E_2}{L}(\rho^+-\mu^+) dx + \int \frac{E_2}{L}(\rho^--\mu^-)
\end{equation}
Indeed, thanks to the periodicity with respect to $x_2$, we first have:
\begin{eqnarray*}
 \int \partial_{x_2} (\rho^+-\mu^+) (\rho^+-\mu^+)dx&= &\int \frac{1}{2}\partial_{x_2} (\rho^+-\mu^+)^2 dx = 0
\end{eqnarray*}

Similarly we have:
\[
 \int \partial_{x_2} (\rho^--\mu^-) (\rho^--\mu^-)dx =0
\]

In the same fashion, with Green's Formula, and using $\operatorname{div} E^\perp = 0$, we have:
\begin{eqnarray*}
 \int E^\perp.\nabla_x (\rho-\mu) (\rho-\mu)dx&=&\frac{1}{2}\int E^\perp.\nabla_x (\rho-\mu)^2dx \\
&=&\frac{1}{2}\int \operatorname{div}(E^\perp (\rho-\mu)^2)dx \\
&=& \frac{1}{2}\left(\int_{x_1=0}E^\perp (\rho-\mu)^2.(-e_1)dx_2 + \int_{x_1=1}E^\perp (\rho-\mu)^2.e_1dx_2\right) \\
&=& 0
\end{eqnarray*}
since ${E_2}=-\partial_{x_2} V =0$ on $x_1=0,L$.

Now, by Lemma \ref{growth} we get:
\begin{eqnarray*}
 \frac{d}{dt} \Vert \rho-\mu \Vert_{L^2}^2 &=&\int -\frac{E_2}{L}(\rho^+-\mu^+) dx + \int \frac{E_2}{L}(\rho^--\mu^-)dx \\
&=& -2\int \frac{E_2}{L} \left(\rho^+-\mu^+\right)dx \left(=  2\int \frac{E_2}{L}(\rho^--\mu^-) \right)
\end{eqnarray*}
We have:
\begin{eqnarray*}
-\int{E_2}\left(\rho^+-\mu^+\right)dx&=& - \int V \partial_{x_2}\left(\rho^+-\mu^+\right)dx + \underbrace{\int \operatorname{div}(V\left(\rho^+-\mu^+\right)e_2)dx}_{=0}\\
&=& \frac{1}{T^+}\int -V\left(\partial_t(\rho^+-\mu^+) + E^\perp.\nabla_x(\rho^+-\mu^+)+\frac{E_2}{L}\right)dx\\
&=& \frac{1}{T^+}\int V\left(-\partial_t(\rho^++\rho^--\mu^+-\mu^-) + E^\perp.\nabla_x(\rho^++\rho^--\mu^+-\mu^-)\right) dx\\
&+&  \frac{1}{T^+} \int -T^- V \partial_{x_2} (\rho^- - \mu^-)dx \\
&=& \frac{1}{T^+}\int V\left(\partial_t\Delta V - E^\perp.\nabla_x\Delta V\right)dx + \frac{T^-}{T^+} \int {E_2}  (\rho^- - \mu^-)dx\\
\end{eqnarray*}
where we have plugged in the equation satisfied by $(\rho^+ - \mu^+)$, by $(\rho^- - \mu^-)$, and also plugged in the Poisson equation.

Finally we have:
\begin{eqnarray*}
\left(1- \frac{T^-}{T^+} \right)\frac{d}{dt} \Vert \rho-\mu \Vert_{L^2}^2 = \frac{2}{LT^+}\int V\left(\partial_t\Delta V - E^\perp.\nabla_x\Delta V\right)dx
\end{eqnarray*}

So that:
\[
\frac{d}{dt} \Vert \rho-\mu \Vert_{L^2}^2 = \frac{2}{L(T^+- T^-)}\int V\left(\partial_t\Delta V - E^\perp.\nabla_x\Delta V\right)dx
\]

To conclude the proof, we observe, using the Dirichlet boundary conditions and the periodicity:
\begin{equation}
\label{onlydir}
 \int V\partial_t\Delta Vdx = -\frac{d}{dt} \frac{1}{2}  \left(\int \vert \nabla V\vert ^2 dx \right)
\end{equation}
and
\begin{eqnarray*}
\int V E^\perp.\nabla \Delta V dx &=& \int V \operatorname{div}(E^\perp\Delta V) dx \\
&=&\underbrace{-\int \nabla V. E^\perp \Delta V dx}_{=0} + \underbrace{\int \operatorname{div} (V E^\perp \Delta V) dx}_{=0}
\end{eqnarray*}
The first term is equal to zero since $E.E^\perp=0$, the second one thanks to the boundary condition on $x_1=0,L$ and to the periodicity with respect to $x_2$.

As a result we have proved that
\[
\frac{d}{dt}\mathcal{E}(t)=0
\]

Actually the computations we have made are rigorously valid only for smooth solutions to (\ref{model}). Nevertheless, these can be justified by smoothing the initial data, and then passing to the weak limit, which entails (\ref{E}).
\end{proof}

\begin{rque}
\label{growth2}
We observe that we have also proved that:
\begin{equation}
\frac{d}{dt}\Vert \rho^+(t)-\mu^+ \Vert_{L^2}^2=\frac{d}{dt} \Vert \rho^-(t) -\mu^- \Vert_{L^2}^2
\end{equation}
which yields:
 \begin{equation}
 \Vert \rho^+(t)-\mu^+ \Vert_{L^2}^2 -  \Vert \rho^-(t) -\mu^- \Vert_{L^2}^2=\Vert \rho^+(0)-\mu^+ \Vert_{L^2}^2 -  \Vert \rho^-(0) -\mu^- \Vert_{L^2}^2.
 \end{equation}
\end{rque}
\begin{rque}

Let us add that the explicit form of the equilibria is crucial in the proof of the theorem. It would not work similarly if we had taken an equilibrium of the form:
\begin{equation}
 \mu(x_1)= \begin{pmatrix}\Phi(x_1) \\ 1-\Phi(x_1) \end{pmatrix} \nonumber
\end{equation}
with $\Phi$ a smooth function. In this case it should be maybe more relevant to use the general Lyapunov functionals method of Arnold \cite{Arn}. 

Likewise one can notice that the proof would have not worked if we had chosen any other boundary condition than Dirichlet. 

\end{rque}
\section{On nonlinear instability}
\label{instability}
What we intend to show now is a property of nonlinear instability in the ``bad curvature'' region when the physical parameters satisfy:
\[
 \frac{T^+- T^-}{L}  < \frac{4}{5 \pi^2}.
\]

 This can be interpreted as a bad confinement property. We first recall that the equilibrium in this case is the following:

\begin{equation}
 \mu^{bad}(x_1)= \begin{pmatrix}1-\frac{x_1}{L}\\ \frac{x_1}{L}\end{pmatrix} \nonumber
\end{equation}

Thanks to the existence of an eigenvalue with maximal positive real part for the linearized operator around $\mu^{bad}$, we can prove  a nonlinear instability result. 
Indeed, using the method introduced by Grenier \cite{Gre}, we are able to pass from the linear spectral instability to the nonlinear instability in the $L^2$ norm. Grenier's method was originally used to prove instability in the $L^2$ velocity norm for Euler ; we show here that this technique can also be adapted to show instability in the $L^2$ vorticity norm.

The drawback of this method is that it requires high regularity on an eigenfunction associated to the dominant eigenvalue, which could be difficult to check in more complicated cases. In our case, we were able to prove such a smoothness in Lemma \ref{reg}.

For the sake of readability we will write $\mu$ instead of $\mu^{bad}$ since there is no risk of confusion.

\begin{thm}
\label{insta}
There exist constants $\delta_0,\eta_1,\eta_2>0$ such that for any $0<\delta<\delta_0$ and any $s\geq0$ there exists a solution $(\rho,E)$ to (\ref{model}) with $\Vert \rho(0)- \mu \Vert_{H^s} \leq \delta$ but such that:
\begin{equation}
\label{instaR}
\Vert \rho(t_\delta)- \mu \Vert_{L^2} \geq \eta_1
\end{equation}
and:
\begin{equation}
\label{instaE}
\Vert E(t_\delta) \Vert_{L^2} \geq \eta_2
\end{equation}
with $t_\delta= O(\vert \log \delta \vert)$.

In particular, $\mu$ is unstable with respect to the $L^2$ norm.
\end{thm}

\begin{rque}
This instability result is complementary to the stability result proved in Corollary \ref{stabi}, since they involve the same $L^2$ norms.
\end{rque}
\begin{rque}
This instability result  can also be obtained by techniques similar to those used by Bardos, Guo and Strauss in \cite{BGS} for the 2D incompressible Euler system, which consist in proving that the solution $\rho$ remains "close" in some norm to a growing mode associated to the maximal growth rate of the linearized operator $M$. 

Using some bootstrap argument introduced by Bardos, Guo and Strauss \cite{BGS}, and a idea of Lin \cite{Lin04a} consisting in:
\begin{itemize}
\item estimating $\Vert E(t_\delta) \Vert_{L^2}$ with Duhamel's formula 
\item then studying the semi-group $e^{-tM}$ (where $M$ is the linearized operator) on $H^{-1}$,
\end{itemize}
we can also prove the exponential growth of the $L^2$ norm of the electric field.

We refer to the paper of Lin \cite{Lin05} which could be adapted to our case with some minor modifications.

Here we will provide an alternative proof  by using Grenier's method (which consists in proving that the solution $\rho$ remains "close" in some norm to the growing mode plus some high order correction) to prove (\ref{instaR}) and finally by using the energy of Theorem \ref{stab} to prove (\ref{instaE}).
\end{rque}

\begin{proof}

We begin with some preliminaries on the linearized operator. Using the same notations as in paragraph \ref{spec}, we consider $M$ the linearized operator around $\mu$ on 
$$Y=\{y_1,y_2 \in H^s([0,L]\times \mathbb{R} / \mathbb{Z})^2, \int (y_1+y_2)dx=0\},\quad \text{for}  \, s \geq 0.$$

For $s=0$, by Corollary \ref{spectral}, we know the existence of an eigenfunction $R$  associated to an eigenvalue $\lambda$ with maximal real part $\operatorname{Re} \lambda$. In addition, by Lemma \ref{reg}, $R$ belongs to any $H^s$, $s\geq 0$. For the sake of simplicity we will assume that $R$ is real and associated to the eigenvalue $\operatorname{Re} \lambda$. In the general case, since the linearized operator is real, the conjugate of $\lambda$ is also an eigenvalue so that one can consider by linearity real-valued growing modes and the following of the proof remains the same.

We recall also that by Corollary \ref{spectral}, for any $\gamma > \operatorname{Re} \lambda$, there is a constant $C(\gamma,s)$ such that for all $s\geq0$:
\begin{equation}
 \Vert e^{-tM}\Vert_{H^s\rightarrow H^s} \leq C(\gamma,s) e^{t \gamma}
\end{equation}

Basically the idea of Grenier is to construct a high order approximation of the nonlinear equation, that is a more precise approximation than the ``usual''  linearized equation.
Indeed instead of showing that $f-\mu$ is close to a well chosen eigenfunction, we show that it is close to the high order asymptotic expansion:
\begin{equation}
 \rho_{app}^{(N)}=\delta u_1 + \sum_{i=2}^N \delta^i u_i
\end{equation}
where $u_1= Re^{\operatorname{Re} \lambda t}$ . Note than for any $s>0$, we have:
$$\Vert u_1\Vert_{H^s} \leq Ce^{\operatorname{Re} \lambda t}.$$

The approximated density $\rho_{app}$ is constructed in order to have the following high order approximation:
\begin{equation}
 \partial_t \rho_{app} + M\rho_{app} + E^\perp_{app}.\nabla_x \rho_{app} =R_{app}
\end{equation}
where $E_{app}=\nabla \Delta^{-1} (\rho_{app}^1+\rho_{app}^2)$ and $R_{app}$ is a remainder satisfying the estimate:
$$\Vert R_{app} \Vert_{H^{L-2N-1}}\leq C \delta^{N+1} \exp{((N+1)\operatorname{Re} \lambda t)}.$$

Let $N\in \mathbb{N}^*$ to be chosen later and take any $S>0$ such that $S>2N+1$ .
We choose also $\theta<1$ (to be fixed later) such that $\frac{1}{2} \geq \frac{\theta}{1-\theta}$ and define $t_\delta$ such that $\theta= \delta \exp{(\operatorname{Re} \lambda t_\delta)}$. 

Now we can construct the $u_j = \begin{pmatrix} u_j^+ \\ u_j^- \end{pmatrix}$ by recursion; we will ensure that for all $1\leq j\leq N$, $\int (u_j^+ + u_j^-)dx=0$ and 
$$\Vert u_j \Vert_{H^{S-j}} \leq C \exp{(j\operatorname{Re} \lambda t)}.$$

Suppose we have $u_j$ for $j\leq k$. Then we define $u_{k+1}$ as the solution of the linear equation:
\begin{equation}
 \partial_t u_{k+1} + Mu_{k+1} + \sum_{j=1}^{k} E^\perp_j.\nabla_x u_{k+1-j} + E^\perp_{k+1-j}.\nabla_x u_{j}=0
\end{equation}
with $E_j=\nabla \Delta^{-1}(u_j^++u_j^-)$ and $u_{k+1}(0,x)=0$ as initial condition. Intuitively, $u_{k+1}$ is chosen in order to counterbalance the non-linear interaction between the previous terms of the expansion.

Thanks to Corollary \ref{spectral} with $\gamma \in ]\operatorname{Re} \lambda, 2\operatorname{Re} \lambda[$,  we get the following estimate:
\begin{eqnarray*}
 \Vert u_{k+1} \Vert_{H^{S-(k+1)}} &\leq& \int_0^t \Vert e^{M(t-s)}(\sum_{j=1}^{k} E^\perp_j.\nabla_x u_{k+1-j} + E^\perp_{k+1-j}.\nabla_x u_{j})\Vert_{H^{L-(k+1)}}ds \\
&\leq&   C \int_0^t  e^{\gamma(t-s)}(\sum_{j=1}^{k} \Vert E^\perp_j \Vert_{H^{S-(k+1)}} \Vert u_{k+1-j}\Vert_{H^{S-k}}+ \Vert E^\perp_{k+1-j}\Vert_{H^{S-(k+1)}}\Vert  u_{j}\Vert_{H^{S-k}})ds \\
&\leq&   C \int_0^t  e^{\gamma(t-s)}  \exp{((k+1)\operatorname{Re} \lambda s)}ds \\
&\leq& C \exp{((k+1)\operatorname{Re} \lambda t)}
\end{eqnarray*}

Note also that since $\frac{d}{dt} \int (u_{k+1}^+ + u_{k+1}^-)dx=0$, we clearly have 
$$\int (u_{k+1}^++u_{k+1}^-)  dx =0.$$

Now we can see that:
\begin{equation}
 \partial_t \rho_{app} + M\rho_{app} + E^\perp_{app}.\nabla_x \rho_{app} =R_{app}
\end{equation}
with $R_{app}=  \sum_{2N\geq j+j'>N} \delta^{j+j'}E^\perp_j.\nabla_x u_{j'}$.  Then, noticing that
for all $t\leq t_\delta$:
$$\delta \exp{(\operatorname{Re} \lambda t)}\leq \theta <1,$$
the following estimate follows:
\begin{equation}
 \Vert R_{app} \Vert_{H^{S- 2N-1}} \leq C_N \delta^{N+1} \exp{((N+1)\operatorname{Re} \lambda t)}
\end{equation}
where $C_N$ is a constant depending only on $N$.

Now we consider the solution $\rho $ to (\ref{model}) such that $\rho(0)-\mu=\rho_{app}(0)$; the equation satisfied by $w= \rho-\mu-\rho_{app}$ is the following:
\begin{equation}
 \partial_t w - \begin{pmatrix}T^+ \partial_{x_2} w^+ \\ T^- \partial_{x_2} w^-  \end{pmatrix} + E^\perp_w.\nabla_x w + E_{app}^\perp.\nabla_x w + E^\perp_w.\nabla_x \rho_{app}= -E^\perp_w.\nabla_x \mu - R_{app}
\end{equation}
with $E_w=\nabla \Delta^{-1}(w^+ - w^-)$.

Then we want to estimate $\Vert w \Vert_{L^2}$ by using some modulated energy inequality. To this end, we multiply by $w$ and integrate with respect to $x$:
\begin{eqnarray*}
 \frac{d}{dt} \Vert w\Vert_{L^2}^2 &\leq &  \int \vert E^\perp_w.\nabla_x \rho_{app} w \vert  dx + \int \vert E^\perp_w.\nabla_x \mu w \vert  dx + \Vert R_{app}\Vert_{L^2}\Vert w \Vert_{L^2} \\
&\leq & ( \Vert\nabla_x \rho_{app} \Vert_{L^\infty} +\Vert \nabla_x \mu \Vert_{L^\infty} )\Vert E^\perp_w \Vert_{L^2} \Vert w \Vert_{L^2} + \frac{1}{2}\Vert w \Vert_{L^2}^2+\frac{1}{2}\Vert R_{app}\Vert_{L^2}^2 \\
&\leq& C\left( (1+ \Vert\nabla_x \rho_{app} \Vert_{L^\infty} )\Vert w \Vert_{L^2}^2 + \Vert R_{app}\Vert_{L^2}^2\right) \\
&\leq& C\left( (1+ \Vert\nabla_x \rho_{app} \Vert_{L^\infty} )\Vert w \Vert_{L^2}^2 + C_N\delta^{2(N+1)} \exp{(2(N+1)\operatorname{Re} \lambda t)} \right)
\end{eqnarray*}

But for $\alpha>0$ such that $2+\alpha<S-N$ and for $t\leq t_\delta$, we can control the Lipschitz norm of $\rho_{app}$:
\begin{eqnarray*}
 \Vert\nabla_x \rho_{app} \Vert_{L^\infty} &\leq& \Vert \rho_{app} \Vert_{H^{2+\alpha}}  \\
& \leq & \sum_{i=1}^N \delta^i \Vert u_i \Vert_{H^{2+\alpha}} \leq \sum_{i=1}^N \delta^i \Vert u_i \Vert_{H^{S-i}}  \\
& \leq & \sum_{i=1}^N \delta^i \exp{(i \operatorname{Re} \lambda t)} \\
&\leq & \sum _{i=1}^N \theta^i \leq \frac{\theta}{1-\theta} \leq \frac{1}{2}
\end{eqnarray*}

Now choose $N$ such that:
\begin{equation}
 N+1> \frac{3C}{4\operatorname{Re}{\lambda}} 
\end{equation}

By Gronwall's lemma we consequently get:
\begin{equation}
 \Vert \rho-\mu-\rho_{app} \Vert_{L^2}= \Vert w\Vert_{L^2} \leq C_N \delta^{N+1} \exp{((N+1)\operatorname{Re} \lambda t)} \leq C_N \theta^{N+1}
\end{equation}

On the other hand we have a bound from below for the $L^2$ norm of $\rho_{app}$, for $t=t_\delta$: 
\begin{eqnarray*}
 \Vert \rho_{app} \Vert_{L^2} &\geq& \delta \Vert u_1 \Vert_{L^2} - \sum_{i=2}^N \delta^i \Vert u_i \Vert_{L^2} \\
&\geq& \delta \exp{(\operatorname{Re}\lambda t_\delta)} - \sum_{i=2}^N \delta^i \exp{(i\operatorname{Re}\lambda t_\delta)}\\
&=& \theta - \sum_{i=2}^N  \theta^i \\
&\geq& \frac{1}{2} \theta
\end{eqnarray*}

Finally we have, for $t=t_\delta$:
\begin{eqnarray*}
 \Vert \rho-\mu \Vert_{L^2} &\geq& \Vert \rho_{app} \Vert_{L^2} - \Vert \rho-\mu-\rho_{app} \Vert_{L^2} \\
&\geq& \frac{1}{2} \theta-C_N \delta^{N+1} \exp{((N+1)\operatorname{Re} \lambda t)} \geq \frac{1}{2}\theta - C_N \theta^{N+1} \\
&\geq& \frac{1}{4} \theta := \eta_1>0
\end{eqnarray*}
if $\theta$ is chosen small enough with respect to $N$. This proves the expected instability result (\ref{instaR}).

Now, in order to prove the exponential growth of the electric field, we use the conservation of the energy proved in Theorem \ref{stab}.
We have for any $t\geq0$:
\begin{equation}
\Vert \rho(t)-\mu \Vert_{L^2}^2 - \frac{1}{L(T^+-T^-)}\int \vert \nabla V(t)\vert ^2 dx \leq \Vert \rho(0)-\mu \Vert_{L^2}^2 - \frac{1}{L(T^+-T^-)}\int \vert \nabla V(0)\vert ^2 dx
\end{equation}
which implies that:
\begin{eqnarray*}
\eta_1^2 \leq \Vert \rho(t_\delta)-\mu \Vert_{L^2}^2& \leq&  \frac{1}{L(T^+-T^-)}\int \vert \nabla V(t_\delta)\vert ^2 dx + \Vert \rho(0)-\mu \Vert_{L^2}^2 \\
&\leq& \frac{1}{L(T^+-T^-)}\int \vert \nabla V(t_\delta)\vert ^2 dx +\delta^2
\end{eqnarray*}
We can consider that $\delta<\delta_0<\eta_1/2$, so that:
\begin{equation}
\int \vert \nabla V(t_\delta)\vert ^2 dx \geq L(T^+-T^-) (\eta_1^2 - \delta_0^2):=\eta_2^2
\end{equation}
This proves (\ref{instaE}).
\end{proof}

By Remark \ref{growth2} we can be a little more precise on the growth of the densities.

\begin{rque}
With the same notations as in the previous theorem, there exists $\eta_3>0$ such that:
\begin{equation}
\Vert \rho^+(t_\delta)- \mu^+ \Vert_{L^2} \geq \eta_3
\end{equation}
and
\begin{equation}
\Vert \rho^-(t_\delta)- \mu^- \Vert_{L^2} \geq \eta_3
\end{equation}
which means that both the hot and the cold plasma are unstable.

\begin{proof}
According to Remark \ref{growth2}, we have
\begin{equation}
\Vert \rho^-(t_\delta)- \mu^- \Vert_{L^2}-\Vert \rho^+(t_\delta)- \mu^+ \Vert_{L^2} =\Vert \rho^-(0)- \mu^- \Vert_{L^2}-\Vert \rho^+(0)- \mu^+ \Vert_{L^2}
\end{equation}
which implies that 
\[
\left\vert \Vert \rho^-(t_\delta)- \mu^- \Vert_{L^2}-\Vert \rho^+(t_\delta)- \mu^+ \Vert_{L^2} \right\vert \leq \delta_0
\]

By Theorem \ref{insta}, we have
\[
\Vert \rho^-(t_\delta)- \mu^- \Vert_{L^2}+\Vert \rho^+(t_\delta)- \mu^+ \Vert_{L^2} \geq \eta_1
\]

Assuming as in the previous proof that $\delta_0 \leq \epsilon_1 /2$, this clearly implies that
\[
\Vert \rho^-(t_\delta)- \mu^- \Vert_{L^2}, \Vert \rho^+(t_\delta)- \mu^+ \Vert_{L^2} \geq \eta_1/4:=\eta_3,
\]
which proves our claim.
\end{proof}
\end{rque}
\section{Conclusion}
\label{conclusion}
We have finally managed to provide a mathematical explanation of stability in the ``good curvature'' region and instability in the ``bad curvature'' region with our simplified nonlinear model. 
In our analysis we have pointed out that large temperature gradients brought nonlinear stability even in the bad curvature region. In other terms, if there is enough heating, there is good confinement: this is the H-mode.

A first natural extension to this work would be to generalize the stability/instability result to the kinetic model (\ref{pre}):
\begin{equation*}
  \left\{
 \begin{array}{ll}
\partial_t f -\frac{1}{2}\vert v\vert^2\partial_{x_{2}}f + E^{\perp}.\nabla_{x} f = 0 \\
  E= -\nabla_{x} V \\
    \displaystyle{-\Delta_{x} V = \int f dv-1}
 \end{array}
\right.
  \end{equation*}
 This shall be the object of a future work. Another important issue is to understand the influence of the gyroaverage operator, that we have neglected in this work.
 
In "real" tokamaks, the next step towards confinement consists in  considering a magnetic field with a variable direction, i.e. $B=B_0 e_\varphi + B_1 e_\theta$. At leading order, particles still follow the magnetic field lines: consequently, with such a twisting field, particles from the ``bad curvature'' region travel every now and then to the ``good curvature'' region. We accordingly expect overall confinement for the plasma.

\begin{center}
 \psfrag{a}{Twisting $B$}
\includegraphics[scale=0.4]{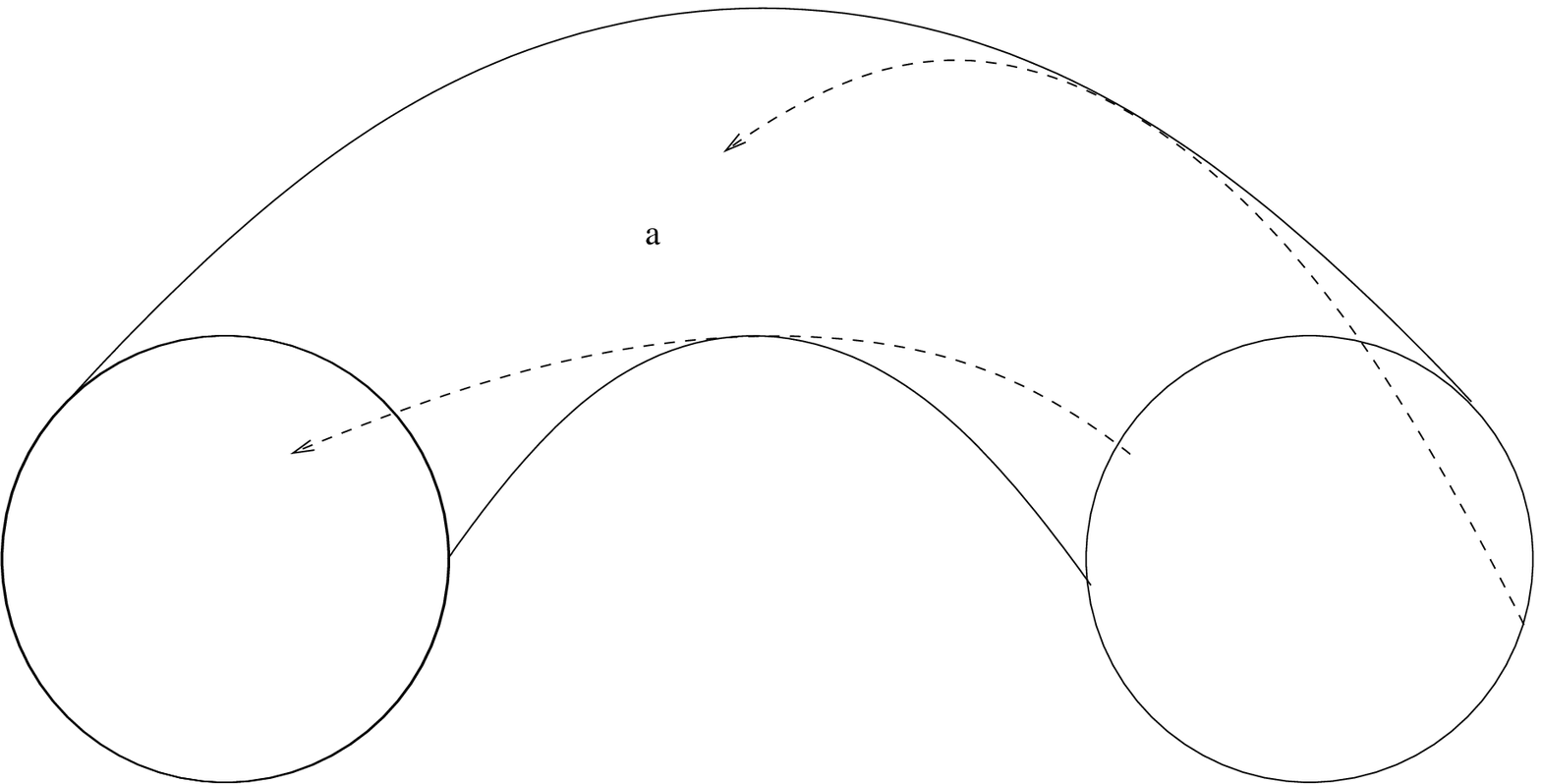}
\end{center}

A very challenging and interesting problem would be to prove overall confinement with such a twisting magnetic field. In this case, the accurate parameter to consider is the so-called safety factor, which stands for the number of times the magnetic field lines twist around the torus the long way for each time they twist around the short way. In ``real'' tokamaks, it has to be chosen with precaution in order to get good confinement properties (see \cite{Wes}).
 But it seems to be much more complicated, since one has to deal with many drifts due to the geometry of the magnetic field. 

\paragraph*{}
Finally let us conclude by mentioning that the analysis of confinement provided in this paper is very naive since it is now well known that there is a loss of confinement in tokamak plasmas, referred to as anomalous transport. Many models have been proposed and intensively studied to justify these phenomena: it would be very interesting to study some of them from the mathematical point of view. 

\bibliographystyle{amsplain}
\bibliography{confinement}

\end{document}